\documentclass[10pt, reqno]{amsart}
\usepackage{amssymb,amsfonts,amsthm,amscd,stmaryrd,dsfont,esint,upgreek}
\usepackage[numbers]{natbib}
\usepackage[mathcal]{euscript}

\usepackage{srcltx}

\usepackage[english]{babel}
\usepackage{graphicx}
\usepackage{amsmath}
\usepackage[T1]{fontenc}
\usepackage[table]{xcolor}
\usepackage{mathrsfs}
\usepackage{exscale}
\usepackage{fancyhdr}
\usepackage{bbm} 
\usepackage{enumitem}
\usepackage{soul}
\usepackage{float}
\usepackage{array}
\usepackage{tabularx}
\usepackage{epic}
\usepackage{eepic}

\makeatletter

\DeclareMathOperator{\lsc}{LSC}
\DeclareMathOperator{\usc}{USC}
\newcommand{\osc}{\operatorname{osc}}
 \newcommand{\sublim}{\operatornamewithlimits{\longrightarrow}}
\providecommand{\norm}[1]{\lVert#1 \rVert}

\renewcommand{\hat}{\widehat}
\renewcommand{\tilde}{\widetilde}

\numberwithin{equation}{section}
\newtheorem{theorem}{Theorem}[section]
\newtheorem{lemma}[theorem]{Lemma}
\newtheorem{prop}[theorem]{Proposition}
\theoremstyle{definition}

\theoremstyle{remark}\newtheorem{step}{Step}           

\newcommand{\R}{\mathbb{R}}
\newcommand{\N}{\mathbb{N}}
\newcommand{\eps}{\varepsilon}
\newcommand{\vphi}{\varphi}
\newcommand{\ds}{\displaystyle}
\newcommand{\udl}{\underline}
\newcommand{\ol}{\overline}
\newcommand{\s}{\beta}

\author[J.-P. DANIEL]{Jean-Paul Daniel}
\address{UPMC Universit\'e Paris 06, UMR 7598 \\
Laboratoire Jacques-Louis Lions, F-75005, Paris, France}
\email{daniel@ann.jussieu.fr}

\title[Quadratic expansions for fully nonlinear parabolic equations]
{Quadratic expansions and partial regularity for fully nonlinear uniformly parabolic equations}
\keywords{partial regularity, fully nonlinear parabolic equations, viscosity solutions.}
\subjclass[2010]{35B65, 35K55, 35D40, 49N60}
\date{\today}
\begin{document}

\begin{abstract}
For a parabolic equation associated to a uniformly elliptic operator, we obtain a~$W^{3, \eps}$~estimate, which provides a lower bound on the Lebesgue measure of the set on which a viscosity solution has a quadratic expansion. The argument combines parabolic $W^{2,\eps}$~estimates with a comparison principle argument. As an application, we show, assuming the operator is~$C^1$, that a viscosity solution is~$C^{2,\alpha}$ on the complement of a closed set of Hausdorff dimension $\eps$ less than that of the ambient space, where the constant~$\eps>0$ depends only on the dimension and the ellipticity.     
\end{abstract}

\maketitle

\section{Introduction}

\subsection{Motivation and statement}

In this paper, we prove a partial regularity result for viscosity solutions of the uniformly parabolic equation
\begin{equation}\label{eq_para_unif_ell_with_g_0} 
\partial_t u + F(D^2u)= 0 
\quad \mbox{in }\Omega \subseteq \R^d \times (-1,\infty).
\end{equation}
We write $u$ as a function of $(x,t)\in \R^d \times [-1,\infty)$ and $D^2u$ denotes the Hessian of $u$ with respect to the~$x$ variable. The operator $F$ is assumed to be uniformly elliptic and to have uniformly continuous first derivatives. 

Recently, Armstrong, Silvestre and Smart \cite{armstrong_smart_silvestre} obtained a partial regularity result for viscosity solutions of the uniformly elliptic equation
\begin{equation} \label{eq_ell_unif_ell}
 F(D^2u)=0 \quad \text{in } U \subseteq \R^d,
\end{equation}
with the same hypotheses on $F$. Specifically, they proved that, for every $0<\alpha <1$, a viscosity solution of \eqref{eq_ell_unif_ell} is $C^{2,\alpha}$ on the complement of a closed set of Hausdorff dimension strictly less than~$d$.

We extend this result to the parabolic setting by showing that the singular set of a solution of~\eqref{eq_para_unif_ell_with_g_0} has Hausdorff dimension at most $d+1 - \eps$, where the constant~$\eps >0$ depends only on the ellipticity of $F$ and $d$. The hypotheses~\ref{f1} and~\ref{f2} are given in the next section. In this paper, H\"older spaces such as $C^{2,\alpha}$ are to be understood in the parabolic sense (e.g., a  parabolic $C^{2,\alpha}$ function may only be~$C^{1,\alpha/2}$ in the time variable: see Section 2.1 for the precise definitions).

\begin{theorem}\label{partpar_reg_th}
 Assume that $F$ satisfies~\ref{f1} and~\ref{f2}. Let $u\in C(\Omega)$ be a viscosity solution of~\eqref{eq_para_unif_ell_with_g_0} in a domain $\Omega  \subseteq \R^{d+1}$. Then there exist a constant $\eps>0$, depending only on $d$, $\lambda$, $\Lambda$ and a closed subset $\Sigma \subseteq \ol \Omega $ of Hausdorff dimension at most $d +1 - \eps$, such that, for every $0<\alpha<1$, the solution $u$ belongs to $C^{2,\alpha}(\Omega \setminus \Sigma)$. 
\end{theorem}

A new difficulty arising in the proof of Theorem~\ref{partpar_reg_th} in the parabolic setting is to obtain an 
analogue of the $W^{3,\eps}$ estimate, an important and useful tool from the regularity theory 
of nondivergence form elliptic equations~\cite{lin_fang_hua,caff_cabre} (see also \cite{armstrong_smart_silvestre}). 
We prove it for viscosity solutions of the uniformly parabolic equation
\begin{equation}\label{eq_para_unif_ell}
\partial_t u + F(D^2u)= g \quad \mbox{in }\Omega \subseteq \R^d \times (-1,\infty), 
\end{equation}
where $g \in C^{0,1}(\Omega)$.
To give the precise statement of this result, we require some notation. 
We denote by $\mathbb{M}_d$ the set of real $d \times d$ matrices. 
The open ball of $\R^d$ centered at $x$ of radius $\rho$ is denoted by $B_\rho(x)$. 
If $x=0$, we simply write $B_\rho$. The following elementary cylindrical domains play a central role in the theory:
for all $\rho>0$ and $x\in \R^d$, we define $Q_\rho(x,t):= B_\rho(x) \times (t-\rho^2,t)$ and denote $Q_\rho:=Q_\rho(0,0)$.

 Now we define, for $u: \Omega \rightarrow \R$, the quantity
 \begin{multline*}\label{def_Psi}
 \Psi(u,\Omega) (x,t) := \inf \left\{ A \geq 0 : \exists (b,p,M)\in \R \times \R^d \times \mathbb{M}_d \text{ s.t. }
\forall (y,s)\in \Omega,  \, s\leq t,  \right.  \\
\left. \left|u(y,s) -u(x,t)- p\cdot (y-x)- b(s-t)- \tfrac{1}{2}(y-x)\cdot M(y-x)\right| \right.  \\
\left.
\leq \tfrac{1}{6}A \left(|x-y|^3+|s-t|^{3/2}\right)  \right\}, 
\end{multline*}
which represents the cubic error in the best quadratic approximation of $u$ at $(x,t)$. 
We emphasize this error is measured globally in $\Omega \cap \{(y,s) : s \leq t\}$ and, of course, ``cubic'' in the time variable means cubic in $t^{1/2}$.

The statement of the parabolic $W^{3,\eps}$~estimate is given by the following theorem.

\begin{theorem}[Parabolic $W^{3,\eps}$ estimate]    \label{W3_eps_para}
 Assume  $F$ satisfies~\ref{f1}, $g \in C^{0,1}(Q_1)$ and $u\in C(Q_1)$ solves~\eqref{eq_para_unif_ell} in~$Q_1$. Then there exist universal constants $C,\eps>0$ such that, for all $\kappa>0$, 
 \begin{multline*} 
  \left| \left\{  (x,t)\in Q_{1/2} \left(0,-\tfrac{1}{4} \right) : \Psi(u,Q_{3/4}) (x,t)>  \kappa \right\} \right| \\
       \leq C \left(\frac{ \kappa}{   \sup_{Q_1} |u|+|F(0)|+\left\|g \right\|_{C^{0,1}(Q_1)}}\right)^{-\eps}. 
 \end{multline*}
\end{theorem}

The elliptic analogue of Theorem~\ref{W3_eps_para} has been used for example to obtain quantitative estimates for the convergence of monotone 
finite difference schemes~\cite{caff_souga_cpam} as well as rates of convergence in homogenization \cite{scott_charlie_quant_stoc, caff_souga_inventiones}. 
We expect Theorem~\ref{W3_eps_para} to have similar applications in the parabolic setting. 

The overall idea of the proof of Theorem \ref{W3_eps_para} is similar to the elliptic case: 
we differentiate the equation to obtain the result from the parabolic $W^{2,\eps}$ estimate. 
In the parabolic case, there is an extra difficulty in controlling the derivative with respect to time. 
Unlike the elliptic case, here we need to use the PDE once more in order to show that by controlling 
all of the spatial derivatives we gain control over~$\partial_tu$. 

The argument for the partial regularity  result is similar to the idea outlined in~\cite{armstrong_smart_silvestre}. 
We apply a result of Wang~\cite{wang_yu_para}, which asserts that any viscosity solution of~\eqref{eq_para_unif_ell_with_g_0} which is sufficiently close to 
a quadratic polynomial must be $C^{2, \alpha}$. This result is a generalization of a result of Savin~\cite{savin_perturb_elliptic} in the elliptic setting. 
Theorem~\ref{W3_eps_para} gives us such quadratic expansions except on a set of lower parabolic Hausdorff dimension.

\medskip

\textbf{Structure of the article:} 
We start by gathering our notation and some preliminary results in Sections~\ref{prelim} and~\ref{prelimini}.   
The proof of Theorem~\ref{partpar_reg_th} is presented in Section~\ref{proof_main_theorem_partial_reg}. 
Section~\ref{w_2_eps_par} is devoted to the derivation of the parabolic $W^{2,\eps}$ estimate. 
Section~\ref{part_par_reg_proof_s} gives the proof of Theorem~\ref{W3_eps_para}.

\section{Preliminaries and proof of the partial regularity result}
\label{hyp_not_proof_prr}

\subsection{Hypotheses and notation}
\label{prelim}
Let $\mathbb S_d \subseteq \mathbb{M}_d$ be the set of symmetric matrices. 
If $M\in \mathbb{M}_d$, $M^\top$~denotes the transpose of $M$. 
Recall that the Pucci extremal operators are defined for constants $0<\lambda \leq \Lambda$ and $M\in \mathbb{S}_d$ by 
\begin{equation*}
 \mathcal{P}^+_{\lambda, \Lambda} (M) := \sup_{\lambda I_d \leq A \leq  \Lambda I_d} - \text{tr} (AM), \qquad \text{and} \qquad
 \mathcal{P}^-_{\lambda, \Lambda} (M) := \inf_{\lambda I_d \leq A \leq \Lambda I_d} - \text{tr} (AM). 
\end{equation*}
A convenient way to write the Pucci operators is
\begin{equation*} \label{puccform}
\mathcal{P}^+_{\lambda, \Lambda} (M) = -\lambda \sum_{\mu_j > 0} \mu_j - \Lambda \sum_{\mu_j < 0} \mu_j 
\quad \mbox{and} \quad \mathcal{P}^-_{\lambda, \Lambda} (M) = -\Lambda \sum_{\mu_j > 0} \mu_j - \lambda \sum_{\mu_j < 0} \mu_j,
\end{equation*}
where $\mu_1, \ldots, \mu_d$ are the eigenvalues of $M$. These operators satisfy the inequalities
\begin{multline*} \label{puccineq}
 \mathcal{P}^-_{\lambda, \Lambda} (M)+ \mathcal{P}^-_{\lambda, \Lambda} (N)
 \leq  \mathcal{P}^-_{\lambda, \Lambda} (M+N)
 \leq  \mathcal{P}^-_{\lambda, \Lambda} (M)+ \mathcal{P}^+_{\lambda, \Lambda} (N) \\
 \leq  \mathcal{P}^+_{\lambda, \Lambda} (M+N)
\leq \mathcal{P}^+_{\lambda, \Lambda} (M)+ \mathcal{P}^+_{\lambda, \Lambda} (N).
\end{multline*}
For the modern theory of nonlinear parabolic equations we refer to~\cite{krylov_para_book,lieberman}. 
A nice introduction to viscosity solutions of nonlinear parabolic equations can be found in~\cite{cyril_luis}. 
Let $Q \subseteq U \times (0,T)$ and $\alpha \in (0,1]$. 
The set of upper and lower semicontinuous functions on $Q$ are denoted by $\usc(Q)$ and~$\lsc(Q)$, respectively.
We will use the following notation.
\begin{itemize}
\item $u\in C^{0,\alpha}(Q)$ means that there exists $C>0$ such that for all $(x,t),(y,s)\in Q$, we have
\begin{equation*}
 \left|u(x,t) - u(y,s)\right| \leq C \left(|x-y|^\alpha + |t-s|^{\frac{\alpha}{2}}\right).
\end{equation*}
In other words, $u$ is $\frac{\alpha}{2}$-H\"older continuous in $t$ and $Du$ is $\alpha$-H\"older continuous in $x$.
 \item $u\in C^{1,\alpha}(Q)$ means that $u$ is $\frac{\alpha +1}{2}$-H\"older continuous in $t$ and $Du$ is $\alpha$-H\"older continuous in $x$.
 \item $u\in C^{2,\alpha}(Q)$ means that $\partial_t u$ is $\frac{\alpha}{2}$-H\"older continuous in $t$ and $D^2u$ is $\alpha$-H\"older continuous in $x$.
 \end{itemize}
Throughout this paper, the nonlinear elliptic operator $F:$ $\mathbb{S}_d \rightarrow \R$ satisfies each of the two following conditions:
\begin{enumerate}[label=\textup{ }{(F\arabic*)}\textup{ },ref={(F\arabic*)}] 
 \item \label{f1}  $F$ is uniformly elliptic and Lipschitz; precisely, we assume that there exist constants $0<\lambda \leq \Lambda$ such that, 
 for every $M,N \in \mathbb{S}_d$,
\begin{equation*}
 \mathcal{P}^-_{\lambda, \Lambda} (M-N) \leq F(M) - F(N) \leq  \mathcal{P}^+_{\lambda, \Lambda} (M-N).
\end{equation*}
 \item \label{f2}  $F$ is $C^1$ and its derivative $DF$ is uniformly continuous, that is, there exists an increasing continuous function 
$\omega : [0,\infty) \rightarrow [0, \infty)$ such that $\omega(0)=0$ and for every $M,N \in \mathbb{S}_d$, 
\begin{equation*}
 |DF(M)-DF(N)|\leq \omega(|M-N|).
\end{equation*}
\end{enumerate}

We call a constant \textit{universal} if it depends only on the dimension $d$ and the ellipticity constants~$\lambda$ and $\Lambda$.
If $U \subseteq \R^d$ and $s< t$, then the parabolic boundary of $U \times (s,t)$ is denoted by
\begin{equation*}
 \partial_p ( U \times (s,t) ) :=  U \times \{s\}  \cup  \partial U \times (s,t) .
\end{equation*}
It is convenient to work with the \emph{parabolic Hausdorff dimension} of a set $E \subseteq \R^{d+1}$,  which is defined by
\begin{multline*} 
  \mathcal{H}_\text{par} (E) := \inf \bigg \{ 0 \leq s < + \infty : 
\text{for all } \delta>0, \text{ there exists a collection } \{Q_{r_j}(x_j,t_j)\}  \\ 
 \text{of cylinders such that }  E \subseteq \bigcup_{j=1}^{+\infty} Q_{r_j}(x_j,t_j) \text{ and } \sum_{j=1}^{+\infty} r_j^s <\delta \bigg\}. 
\end{multline*}
The relationship between the parabolic Hausdorff measure $\mathcal{H}_\text{par}(E)$ and the classical Hausdorff measure 
  $\mathcal{H}(E)$ is given by 
\begin{equation}\label{link_para_classical_hausdorff_measure}
 2\, \mathcal{H}(E) - d   \leq  \mathcal{H}_\text{par} (E) \leq  \mathcal{H}(E) +1.
\end{equation}
The reader is referred to \cite{sweezy} for further details about the parabolic framework. 
We remark that $ \mathcal{H}_\text{par}(\R^{d +1})=d+2$.

\subsection{Preliminaries} \label{prelimini}
First we recall an interior $C^{1,\alpha}$ regularity result for solutions of~\eqref{eq_para_unif_ell}. 
\begin{prop}[{\cite[Section~4.2]{wang_para_II}}] 
\label{int_reg_parabolic} 
 If $u$ is a viscosity solution of~\eqref{eq_para_unif_ell} in $Q_1$, then $u\in C^{1,\alpha}(\ol Q_{1/2})$ for some universal $0<\alpha<1$. 
 Moreover, there exists some universal constant $C$ such that
 \begin{equation*}
  \sup_{\ol Q_{1/2}} |Du| \leq C \left(\sup_{Q_1}|u| +|F(0)|+\left\| g \right\|_{C^{0,1}(Q_1)} \right).
 \end{equation*}
\end{prop}
Moreover, it is natural in the parabolic framework to introduce some other sets called parabolic balls. 
Precisely, given $(x,t)\in \R^{d+1}$, we define \textit{parabolic balls} of opening $\theta>0$ and height $h\geq 0$ to be closed subsets of the following~form
 \begin{equation*}
  G^{}_{\theta,h}(x,t):= \left\{ (y,s)\in \R^{d+1} : \theta |y-x|^2 \leq  s-t \leq h \right\}  
 \end{equation*}  and
 \begin{equation*}
  G^{-}_{\theta,h}(x,t):= \left\{ (y,s)\in \R^{d+1} : \theta |y-x|^2 \leq  t-s \leq  h \right\}.  
 \end{equation*}
By direct computation, 
\begin{equation}\label{volume_par_ball}
 |G^{}_{\theta,h}(x,t)| = |G^{-}_{\theta,h}(x,t)|= \frac{2 \omega_d}{d+2} h^{1+d/2} \theta^{-d/2}, 
\end{equation}
where $\omega_d$ is the volume of the unit ball in $\R^d$.
We next collect some standard material about parabolic balls. 
The following lemma is a slight modification of \cite[Lemma 2.2]{wang_yu_para}. 

\begin{lemma} [Y. Wang, \cite{wang_yu_para}]
\label{construct_cylinder_pb} Let $\theta \geq 3/4$. 
For all $(x,t) \in G^{-}_{\theta,h_0}(x_0,t_0)$ and $0<h\leq t_0-t$, 
there exists a cylinder $Q_r(x_2,t_2)$ such that 
\begin{equation*}\label{borne_inf_r_h}
t_2=t+\frac{h}{2}, \quad  r\geq \frac{1}{\nu} \sqrt{\frac{h}{\theta}} \quad \text{with} \quad \nu:= \frac{4}{\sqrt{2}-1}, 
\end{equation*}
which satisfies the three following properties: 
 \begin{enumerate}[label=\textup{(}{P\arabic*}\textup{)},ref=\textup{(}{P\arabic*}\textup{)}]
  \item \label{inclusion_Kr_pb} 
  $\ol Q_r(x_2,t_2) \subseteq G^{}_{\theta,h}(x,t) \cap G^{-}_{\theta,h_0}(x_0,t_0) 
    \cap \{(y,s)  :t+h/4 \leq s \leq t +h/2 \}  $. 
  \item \label{ratio_good_cylr_pb} $|Q_r(x_2,t_2)| / |G^{}_{\theta,h}(x,t)| \geq \eta_0$ where $\eta_0$ depends only on $d$ and $\theta$. 
  \item \label{inclusion_ball_contact_set} If $(z,t)\in Q_{r/4}(x_2,t_2)$, 
  then for every $(y,s) \in G^{-}_{1/2,r^2/16} \left(z,t - \tfrac{r^2}{16}\right)$, 
  \begin{equation*}
  G^{}_{1/2,t-s}(y,s) \subseteq Q_r(x_2,t_2).  
  \end{equation*}
 \end{enumerate}
\end{lemma}

\begin{proof}
The reader is referred to \cite[Lemma 2.2]{wang_yu_para} for~\ref{inclusion_Kr_pb} and \ref{inclusion_ball_contact_set}.
For~\ref{ratio_good_cylr_pb}, we compute  
\begin{equation*}
|Q_r (x_2,t_2)| = r^{d+2}
\geq  \left(   \tfrac{\sqrt{2}-1}{4}  \right)^{d+2} \theta^{-d/2-1} h^{1+d/2}
= \tfrac{d+2}{2 \omega_d} \left(   \tfrac{\sqrt{2}-1}{4} \right) ^{d+2} \theta^{-1} |G_{\theta,h}(x,t)|.
\end{equation*} 
It suffices to set $ \eta_0:=\frac{d+2}{2^{d+3} \omega_d} \left(\tfrac{\sqrt{2}-1}{4}  \right)^{d+2} \theta^{-1}$ 
to get the desired estimate. 
\end{proof}

Finally, we recall a Vitali-type covering lemma for parabolic balls which was already used and proved in~\cite{wang_yu_para} (and essentially follows from the standard argument for Vitali's covering lemma).
It is a convenient alternative in the parabolic setting to the ``stacked'' estimate lemma and 
the Calder\'on-Zygmund decomposition (see~\cite{cyril_luis}). 
Given $(x,t) \in \R^{d+1}$ and $h>0$, we define the parabolic ball 
$\widehat{G}^{}_{\theta,h}(x,t) \supseteq G^{}_{\theta,h}(x,t)$ by
\begin{equation*}
\widehat{G}^{}_{\theta,h}(x,t) := G^{}_{\hat \theta,4h}(x,t - 3h) \qquad \text{with} \qquad \hat \theta:=\frac{\theta}{(\sqrt{2}+1)^2}.
\end{equation*}
We observe from~\eqref{volume_par_ball} that the following ratio is a universal constant depending only on the dimension:
\begin{equation} \label{ratio_pb_vitali}
 \eta : = \frac{|G^{}_{\theta,h}(x,t)|}{|\widehat{G}^{}_{\theta,h}(x,t)|} = 4^{-(1+d/2)} (\sqrt{2}+1)^{-d}.
\end{equation}
\begin{lemma}[Vitali's lemma for parabolic balls]\label{vitali_par_balls}
Assume that $E \subseteq \R^{d+1}$ is bounded and $h : E \to \R$ is positive. Consider  the following collection of parabolic balls: 
\begin{equation*}
   \left\{ G^{}_{\theta,h(x,t)} (x,t) :  (x,t) \in E \right\}. 
\end{equation*}
If $\sup \{ h(x,t) : (x,t) \in E \}<\infty$, 
we can extract a countable subcollection $\{G_{\theta,h(x_i,t_i)} (x_i,t_i) : i\in \N \}$ of disjoint parabolic balls such that
 \begin{equation*}
  E \subseteq \bigcup_{i\in \N} \widehat G^{}_{\theta,h(x_i,t_i)} (x_i,t_i).
 \end{equation*}
\end{lemma}

Finally, we give the statement of a proposition required to obtain the partial parabolic result. 
This proposition was obtained by Y. Wang \cite{wang_yu_para} and is the parabolic analogue of a first result of Savin~\cite{savin_perturb_elliptic}. 
It gives $C^{2,\alpha}$~regularity for flat viscosity solutions of uniformly parabolic equations.
 Roughly speaking, it  states that a viscosity solution of a uniformly parabolic equation 
 that is sufficiently close to a quadratic polynomial is, in fact, a classical solution. 

 \begin{prop}[Y. Wang, \cite{wang_yu_para}]\label{flatreg_adapted_par}
 Suppose in addition to \ref{f1}--\ref{f2} that $F(0)=0$. 
Suppose that $0 < \alpha < 1$ and $u\in C(Q_1)$ is a solution of \eqref{eq_para_unif_ell_with_g_0} in $Q_1$. 
Then there exists a constant~$\delta_0>0$ depending only on 
the ellipticity constants $\lambda$ and $\Lambda$, the dimension $d$, the modulus of continuity $\omega$, and $\alpha$, such that
\begin{equation*}
\sup_{Q_1} |u| \leq \delta_0 \quad \text{implies that } \quad u\in C^{2,\alpha}(Q_{1/2}). 
\end{equation*}
\end{prop}

\subsection{Proof of Theorem~\ref{partpar_reg_th}} 
\label{proof_main_theorem_partial_reg}
The strategy of the proof is the following. By a covering argument, we can cover the singular set by parabolic balls centered in points for which~$\Psi$ presents large values. 
This is allowed by Lemma~\ref{lemma_flat_reg} that shows that $u$ is not~$C^{2,\alpha}$ close to $(x_0,t_0)$ implies~$\Psi(x_0,t_0)$ is large. 
Then the parabolic $W^{3,\eps}$~estimate given by Theorem~\ref{W3_eps_para} provides an upper bound on the size of the set of the bad points.

\begin{lemma}\label{lemma_flat_reg}
Suppose $u\in C(Q_1)$ solves \eqref{eq_para_unif_ell_with_g_0} in $Q_1$ and $0<\alpha<1$.
There is a universal constant $\delta_\alpha>0$, depending on $d$, $\lambda$, $\Lambda$, $\omega$ and $\alpha$
such that for every $(y,s_0)\in Q_{1/2}\left(0, -\tfrac{1}{4}\right)$ and $0<r<1/20$, 
 \begin{equation*}
 \left\{ \Psi(u,Q_{3/4}) \leq r^{-1} \delta_\alpha \right\} \cap Q_r(y,s_0) \neq \emptyset 
\quad \text{ implies that } \quad u\in C^{2,\alpha}(Q_r(y,s_0 - r^2)).
 \end{equation*}
\end{lemma}

\begin{proof}Let $\delta>0$ to be adjusted. 
Suppose that $0<r<1/20$, $(y,s_0)\in Q_{1/2}\left(0, -\tfrac{1}{4} \right)$ and $(z,s)\in Q_r(y,s_0)$ is such that
\begin{equation*}
\Psi(u,Q_{3/4})(z,s) \leq r^{-1} \delta.
\end{equation*}
Then there exist $b\in \R$, $p\in \R^d$ and $M\in \mathbb{M}_d$ such that, for every $(x,t)\in Q_{3/4}$ such that $t\leq s$, 
\begin{multline}\label{ineq_Upsi_preg}
|u(x,t)  - u(z,s) -  p\cdot (x-z) -  b(t-s) - \tfrac{1}{2} (x-z) \cdot M (x-z) | \\
\leq \tfrac{1}{6} r^{-1}  \delta  (|x-z|^3 + |s-t|^{3/2}). 
\end{multline}
Replacing $M$ by $\frac{1}{2} (M+M^\top)$, we may assume that $M\in \mathbb{S}_d$. 
Since $u$ is a viscosity solution of \eqref{eq_para_unif_ell_with_g_0}, it is clear that
\begin{equation*}
b+F(M)= 0. 
\end{equation*}
For $(x,\tau)\in Q_1$, define 
\begin{equation*}
 v(x, \tau):=\frac{1}{16r^2} 
 \left( u(z+4rx,s+ 16 r^2 \tau) - u(z,s) - 4r p\cdot x -16 b r^2 \tau - 8 r^2 x\cdot Mx  \right). 
\end{equation*}
Noticing that $(z+4rx,s+ 16 r^2 \tau)\in Q_{3/4}$ for $(x,\tau)\in Q_1$, the inequality \eqref{ineq_Upsi_preg} implies that 
\begin{equation*}
\sup_{Q_1} |v(x,\tau)|\leq \frac{4}{3}\delta. 
\end{equation*}
Define the operator $\widetilde F(N):=b+F(N+M)$, and observe $\widetilde F$ satisfies \ref{f1} and \ref{f2}, with the same ellipticity constants $\lambda$, $\Lambda$ and modulus $\omega$,  
and $\tilde F(0)=b + F(M)=0$. It is clear that $v$ is a solution of 
\begin{equation*}
 \partial_tv+\tilde F(D^2v)=0, \quad \mbox{in } Q_1.
\end{equation*}
Let $\delta_0>0$ be the universal constant in Proposition~\ref{flatreg_adapted_par}, which also depends on $\alpha$. 
Suppose that 
\begin{equation*}
 \delta <\frac{3}{4} \delta_0.
\end{equation*}
Then Proposition~\ref{flatreg_adapted_par} yields that $v \in C^{2,\alpha}(Q_{1/2})$, from which we deduce that $u \in C^{2,\alpha}(Q_{2r}(z,s))$. 
Since $Q_r(y,s_0 - r^2) \subseteq Q_{2r}(z,s)$, the proof is complete. 
\end{proof}

We now give the proof of the first main result. 

\begin{proof}[Proof of Theorem \ref{partpar_reg_th}]
We assume without loss of generality that $F(0)=0$. 
By a standard covering argument, we may fix $0<\alpha<1$ and assume that 
$\Omega =Q_1$, $u\in C(Q_1)$ is bounded, and it suffices to show that
\begin{equation*}
 u \in C^{2, \alpha}(V \setminus \Sigma)
\qquad \text{with} \qquad
V := B_{9/20} \times \left(-\frac{1}{2}, -\frac{1}{4} -\frac{3}{800}\right).
\end{equation*}
for a set $\Sigma \subseteq \ol V$ with  $ \mathcal{H}_\text{par}(\Sigma) \leq d +2 - \eps$.
Since, for every $\s>0$, the operator $F_\s (M):= \s^{-1} F(\s M)$ satisfies 
both \ref{f1} and \ref{f2} with the same constants $\lambda$ and $\Lambda$ but a different modulus $\omega$ 
and the constant $\eps$ that we obtain does not depend on $\omega$, 
we may therefore assume without loss of generality that $\sup_{Q_1}|u|\leq 1$. Let $\Sigma \subseteq V$ denote the set
\begin{equation*}
 \Sigma: =\left\{(x,s)\in V : u\notin C^{2,\alpha}\left(Q_r \left(x,s +\tfrac{1}{2} r^2 \right)\right) \text{ for every }r>0\right\}.
\end{equation*}
Notice that $\Sigma$ is closed, and thus compact. 
Fix $0<r<1/20$. According to the Vitali covering theorem for parabolic cylinders~\cite[Lemma~7.8]{lieberman}, 
there exists a finite collection $\{Q_r(x_i,s_i+\tfrac{1}{2}r^2)\}_{1\leq i \leq m}$ of disjoint parabolic cylinders of radius $r$, 
with centers $(x_i,s_i) \in \Sigma$, such that
\begin{equation*}
 \Sigma \subseteq \bigcup_{i=1}^m Q_{5r}\left(x_i,s_i+\tfrac{25}{2}r^2 \right).
\end{equation*}
Since $(x_i,s_i) \in \Sigma$, according to Lemma \ref{lemma_flat_reg}  there exists a constant $\delta$ such that
\begin{equation*}
 \Psi(u,Q_{3/4})(y,\tau)>r^{-1}\delta \quad \text{ for every } (y,\tau)\in \bigcup_{i=1}^m Q_r\left(x_i,s_i+\tfrac{3}{2}r^2\right). 
\end{equation*}
Applying Theorem~\ref{W3_eps_para} 
to $Q_r\left(x_i,s_i+\tfrac{3}{2}r^2\right) \subseteq Q_{1/2}\left(0, - \tfrac{1}{4}\right)$, we deduce that
\begin{equation*}
 mr^{d+2} = m \frac{|Q_r|}{|Q_1|} \leq C\left(r^{-1}\delta\right)^{-\eps}
\end{equation*}
for some universal constants $C,\eps>0$. Therefore, 
\begin{equation*}
 \sum_{i=1}^m   \left(5r\right)^{d+2-\eps} \leq 5^{d +2-\eps}m r^{d+2-\eps} \leq 5^{d +2-\eps} C  \delta^{-\eps} <+\infty.  
\end{equation*}
In particular, this implies that
\begin{equation*}
 \mathcal{H}_\text{par}(\Sigma) \leq d+2-\eps.
\end{equation*} 
By using \eqref{link_para_classical_hausdorff_measure}, we get 
\begin{equation*}
 \mathcal{H}(\Sigma) \leq \frac{1}{2}\left(d +d+2-\eps \right)=d +1 - \frac{\eps}{2}. \qedhere
\end{equation*} 
\end{proof}

 \section{Parabolic $W^{2,\eps}$  estimate}  \label{w_2_eps_par}

In this section, we state and prove the parabolic $W^{2,\eps}$  estimate associated to~\eqref{eq_para_unif_ell}. It will be useful to prove Theorem~\ref{W3_eps_para} in Section~\ref{part_par_reg_proof_s}. This result is essentially well-known but we give the argument here for the sake of completeness. 

The key argument to prove Proposition~\ref{W_2eps_para} relies on a measure estimate on small parabolic balls stated in Lemma~\ref{est_measure_small_pb_bd}. Its proof consists in a suitable parabolic Aleksandrov-Bakelman-Pucci~(ABP) inequality and a comparison principle achieved with a barrier function. Then we derive an induction relation in Lemma~\ref{measure_estimate_macroscopic_ball} by Lemma~\ref{vitali_par_balls}. Finally some classical arguments permit to obtain Proposition~\ref{W_2eps_para}.

\smallskip

Let $\Omega \subseteq \R^{d +1}$ and $u, -v \in \lsc(\Omega)$. We say that $v$ \emph{touches $u$  
from below at~$(x,t)\in \Omega$} if
\begin{equation*}
 \begin{cases}
 v(z,\tau) \leq u(z,\tau) , & \text{for } (z,\tau)\in \Omega \text{ and } \tau \leq t,  \\ 
 v(x,t) = u(x,t) .
 \end{cases}
\end{equation*}
We say that \emph{$u$ touches $v$ from above at~$(x,t)\in \Omega$} if $v$ touches $u$ at~$(x,t)\in \Omega$ from below.
Let $(y,s) \in \R^{d+1}$. A polynomial $P$ is called a \emph{concave paraboloid} of opening~$\kappa>0$ if 
\begin{equation*}
 P_{y,s; \kappa}(z,\tau)=  - \tfrac{\kappa}{2} |z-y|^2 + \kappa (\tau -s). 
\end{equation*}
Similarly, a polynomial $P$ is called a \emph{convex paraboloid} of opening $\kappa>0$ if 
\begin{equation*}
 P_{y,s; \kappa}(z,\tau)=  \tfrac{\kappa}{2} |z-y|^2 - \kappa (\tau -s).
\end{equation*}

To state the estimate, we require some notation. 
Given a domain $\Omega \subseteq \R^{d+1}$, and a function~$u\in \lsc(\Omega)$, 
define the quantity
\begin{multline*}
\udl \Theta(x,t)  = \udl \Theta (u,\Omega) (x,t) 
:= \inf \left\{A \geq 0 : \exists p\in \R^d \  \text{ s.t. } \  \forall (y,s)\in \Omega, \, s \leq t , \right. \\
\left. u(y,s) \geq u(x,t)  +p\cdot (y-x) - A  \left( \tfrac{1}{2} |x-y|^2 + (t-s) \right)  \right\} .
\end{multline*}
 Similarly, for~$u\in \usc(\Omega)$, 
\begin{multline*}
\ol \Theta(x,t) = \ol \Theta (u,\Omega) (x,t) 
 := \inf \left\{A \geq 0 : \exists p\in \R^d \  \text{ s.t. } \  \forall (y,s)\in \Omega ,  \, s \leq t,  \right. \\
 \left.   u(y,s) \leq u(x,t) +p\cdot (y-x) + A  \left( \tfrac{1}{2} |x-y|^2 + (t-s) \right)   \right\} , 
\end{multline*}
and, for $u\in C(\Omega)$,  
\begin{equation*}
\Theta(x,t)=\Theta (u,\Omega)(x,t) :=\max \left\{ \udl \Theta (u,\Omega)(x,t), \ol \Theta (u,\Omega)(x,t) \right\}. 
\end{equation*}
The quantity $\udl \Theta(x,t)$ is the minimum curvature of any paraboloid that touches $u$ from below at~$(x,t)$. 
If $u$ cannot be touched from below at $(x,t)$ by any paraboloid, then $\udl \Theta(x,t)=+\infty$. 
A similar statement holds for $\ol \Theta(x,t)$, where we touch from above instead. 
Moreover, a function $u$ is~$C^{1,1}$ on a closed set $\Gamma \subseteq \Omega$ if and only if~$u$ has tangent paraboloids 
from above and below with respect to $\Omega$ at each point of~$\Gamma$.

The form of the $W^{2,\eps}$ estimate we need is given by the following proposition. 
\begin{prop}\label{W_2eps_para}
 If $u \in \lsc(Q_1)$ and $L\geq 0$ satisfy the inequality
 \begin{equation*}
  \partial_tu+ \mathcal{P}^+_{\lambda, \Lambda} (D^2u) \geq  - L \quad \text{in }Q_1,
 \end{equation*}
then for all $\kappa> 0$, 
\begin{equation*} 
|\{ (x,t)\in Q_{1/2}\left( 0, -\tfrac{1}{4}\right) : \udl \Theta (u,Q_1)(x,t)>   \kappa \}| 
\leq  C \left(\frac{\kappa}{\sup_{Q_1} |u|+ L } \right)^{-\eps}, 
 \end{equation*}
 where the constants $C$ and $\eps>0$ are universal.
\end{prop}
We emphasize here that $\udl \Theta (u,\Omega)$ is defined in terms of quadratic polynomials that touch $u$ at~$(x,t)$ and stay below $u$ in the domain $\Omega \cap \{(y,s) : s\leq t\}$, which is full in space and restricted to global times less than $t$.

Instead of working with the sets $\{\udl \Theta\leq \kappa \kappa\}$, we are going to consider some new sets $A_\kappa$ for $\kappa>0$. 
We are inspired from the elliptic definition introduced by Savin~\cite{savin_perturb_elliptic} and recently 
also used by Armstrong and Smart in~\cite{scott_charlie_regstoch_out_unif}. In the parabolic setting, define, for every $\kappa>0$, 
\begin{multline} \label{def_A_op}
 A_\kappa:= \bigg\{ (x,t)\in Q_1 : \exists (y,s) \in B_1 \times (-1,t] \ \text{ s.t.} \    
 u(x,t) - \inf_{Q_1} u +\kappa \left(\tfrac{1}{2} |x-y|^2- (t-s) \right)  \\
=\inf_{\substack{(z,\tau) \in Q_1,\\ \tau \leq t} } \left( u(z,\tau) - \inf_{Q_1} u  
+\kappa \left(\tfrac{1}{2} |z-y|^2-(\tau-s) \right) \right) =0  \bigg\}.
\end{multline}
It is important to notice that $s$ takes part in the definition of $A_\kappa$ only to adjust the infimum to be equal to zero. Moreover, the definition of $A_\kappa$ given above is adapted to the domain $Q_1$. It is clear how to change this definition of $A_\kappa$ to deal with more general domains.

The next lemma gathers some properties about the sets $A_{\kappa}$. In particular, the link between $A_\kappa$ and $\udl \Theta$ is precised. 

\begin{lemma}\label{comparison_Aop_Theta} 
Let $u\in \lsc(\Omega)$, $\kappa>0$ and $A_\kappa$ be defined by \eqref{def_A_op}. Then we have
\begin{equation*}
 A_\kappa \subseteq \left\{(x,t) \in Q_1 : \udl \Theta(u,Q_1)(x,t) \leq \kappa  \right\}.
\end{equation*}
 Moreover, for all  $0<\kappa_1 \leq \kappa$, we have $A_{\kappa_1} \subseteq A_{\kappa}$.
\end{lemma}

\begin{proof}
Let $(x,t) \in A_\kappa$. Then $(x,t)\in Q_1$ and there exists $y \in B_1$ 
such that for all $(z,\tau) \in Q_1$, $\tau \leq t$, we have
\begin{equation*}
 u(z,\tau)+\kappa \left(\tfrac{1}{2} |z-y|^2-\tau \right) \geq   u(x,t)+\kappa \left(\tfrac{1}{2} |x-y|^2- t \right) .
\end{equation*}
After rearranging the terms, we get
\begin{equation*}
 u(z,\tau) \geq   u(x,t)+\kappa \left(\tfrac{1}{2} (|x-y|^2 -|z-y|^2 ) - (t - \tau) \right) .
\end{equation*}
 An algebraic manipulation yields $ |x-y|^2 -|z-y|^2 = - |x-z|^2  + 2 (y-x) \cdot (z-x)$.
If we choose $p= \kappa (y-x) \in \R^d$, the last inequality can be written in the form
\begin{equation*}
 u(z,\tau) \geq   u(x,t)+ p \cdot (z-x)  - \kappa \left(\tfrac{1}{2} (|x-z|^2    + (t - \tau) \right),
\end{equation*}
for all $(z,\tau) \in Q_1$, $\tau \leq t$, which  gives precisely $\udl \Theta(u,Q_1)(x,t) \leq \kappa$. \\
For the second assertion, let $(x_1,t_1) \in A_{\kappa_1}$ and $\kappa>\kappa_1$. Hence there exist $y_1 \in B_1$ and  $-1< s_1 \leq t_1$ such that
\begin{multline}\label{def_xt_Aop1}
u(x_1,t_1) - \inf_{Q_1} u +\kappa_1 \left(\tfrac{1}{2} |x_1-y_1|^2- (t_1-s_1) \right)  \\
=\inf_{\substack{(z,\tau) \in Q_1,\\ \tau \leq t} } \left( u(z,\tau) - \inf_{Q_1} u  
+\kappa_1 \left(\tfrac{1}{2} |z-y_1|^2-(\tau-s_1) \right) \right) =0 . 
\end{multline}
For $(y,s) \in \R^{d+1}$, let $P$ be the paraboloid given by 
\begin{equation}\label{def_P_correction}
 P(z,\tau):= \kappa \left(\tfrac{1}{2} |z-y|^2-(\tau-s) \right)
 - \kappa_1 \left(\tfrac{1}{2} |z-y_1|^2-(\tau-s_1) \right).
\end{equation}
Assume that we have shown there exists $(y,s) \in  B_1 \times (-1,t_1]$ such that, for all $(z,\tau)\in \R^{d+1}$,  
  \begin{equation} \label{form_P_special}
  P(z,\tau)=(\kappa - \kappa_1) \left(\tfrac{1}{2}  |z - x_1|^2 - (\tau- t_1)\right).
  \end{equation} 
For this particular choice, the decomposition 
\begin{multline*}  
 u(z,\tau) - \inf_{Q_1} u +\kappa \left(\tfrac{1}{2} |z-y|^2-(\tau-s) \right) \\
= u(z,\tau) - \inf_{Q_1} u +\kappa_1 \left(\tfrac{1}{2} |z-y_1|^2-(\tau-s_1) \right)  +P(z,\tau)
\end{multline*}
implies the result by using \eqref{def_xt_Aop1}, $P\geq 0 $ on $B_1 \times (-1, t_1]$ and $P(x_1,t_1)=0$ by~\eqref{form_P_special}.  To obtain the assertion, it remains to show~\eqref{form_P_special}.
By completing the square in~\eqref{def_P_correction}, the polynomial $P$ can be written in the form 
\begin{equation*}
P(z,\tau)=(\kappa - \kappa_1) \left(\tfrac{1}{2}  |z - \ol z(y,s)|^2 - (\tau- \ol \tau(y,s))\right),
\end{equation*}
with 
 \begin{equation*}
  \left\{   
     \begin{aligned}
 &\ol z(y,s) :=  \frac{1}{\kappa - \kappa_1} \left(\kappa y - \kappa_1 y_1 \right), \\
& \ol \tau(y,s) := \frac{1}{\kappa - \kappa_1}  \left( \kappa s - \kappa_1 s_1 +\frac{1}{2}\frac{\kappa \kappa_1}{\kappa - \kappa_1} |y-y_1|^2    \right).
 \end{aligned}
 \right.
\end{equation*}
We choose $(y,s)$ to impose the condition  $(\ol z(y,s), \ol \tau(y_1,s_1))=(x_1,t_1)$. This leads to select
\begin{equation*}
 y= \left(1 - \frac{\kappa_1}{\kappa}\right) x_1 +\frac{\kappa_1}{\kappa} y_1 
 \quad \text{and} \quad 
  s=t_1 - \frac{\kappa_1}{\kappa} \left( t_1 - s_1 + \frac{1}{2}  |x_1-y_1|^2 \left(1 -  \frac{\kappa_1}{\kappa}\right) \right).  
\end{equation*}
It is clear by convexity that $y \in B_1$.
To show that $s\in (s_1,t_1]$, notice $t_1 - s_1 -\frac{1}{2} |x_1-y_1|^2\geq 0$ by \eqref{def_xt_Aop1} and use $\kappa>\kappa_1$ in the expression above. 
\end{proof}

Now we recall the standard tool in the theory of viscosity solutions (see~\cite{user_s_guide} for further details).
We denote the infimal convolution of $u\in \lsc(Q_1)$ by 
\begin{equation*}
 u_\eps(x,t)= \inf_{(z,\tau)\in Q_1} \left( u(z,\tau)+\frac{2}{\eps} (|z-x|^2+(\tau - t)^2)   \right).
\end{equation*}
Moreover, if $f\in C(Q_1)$ and 
\begin{equation*}
 \partial_tu+\mathcal{P}^+_{\lambda, \Lambda}(D^2u) \geq f \quad \text{in }Q_1,
\end{equation*}
then there exist a sequence of functions $f_\eps \in C(Q_1)$ which converges locally uniformly to $f$ respectively, as $\eps \rightarrow 0$, 
such that $u_\eps$ satisfies
\begin{equation*} 
 \partial_t u+\mathcal{P}^+_{\lambda, \Lambda}(D^2u) \geq f_\eps \quad \text{in }Q_{1-r_\eps}(0,T_\eps), 
\end{equation*}
when $r_\eps \rightarrow 0$ and $T_\eps \rightarrow 0$ as $\eps \rightarrow 0$.
The function $u_\eps$ is more regular than $u$ and, in particular, is semiconcave. It is a good approximation to $u$
in the sense that $u_\eps \rightarrow u$ locally uniformly in $Q_1$ as~$\eps \rightarrow 0$.
For us, the main utility of these approximations is the semiconcavity of $u_\eps$. 
If $u_\eps$ can be touched from below by a smooth function~$\vphi$ at some point$(y,s)\in Q_1$, then $u_\eps$ is $C^{1,1}$
at $(y,s)$, with norm depending only on~$\eps$ and~$|D^2\vphi(y,s)|$ and $\partial_t\vphi(y,s)$.

The following lemma is the form of the ABP inequality we are going to use.
\begin{lemma}\label{lemma_abp_measure_inf}
 Assume that $L>0$ and $u \in \lsc(Q_1)$ satisfy
 \begin{equation*}
  \partial_t u + \mathcal{P}^+_{\lambda, \Lambda} \left( D^2u \right) \geq  - L \text{ in } Q_1.
 \end{equation*}
  Suppose that  $a>0$ and $V  \subseteq \R^{d+1}$ is compact such that,
 for each $(y,s) \in V$, there exists $(x,t)\in Q_1$ such that
 \begin{multline}\label{def_point_inf}
u(x,t) - \inf_{Q_1}u +\tfrac{a}{2}|x-y|^2 - a(t-s) \\
= \inf_{(z,\tau)\in Q_1, \tau \leq t} \left(u(z,\tau) - \inf_{Q_1}u +\tfrac{a}{2}|z-y|^2 - a (\tau -s) \right)=0.
 \end{multline}
Let $W:=\{(x,t) \in Q_1 : \eqref{def_point_inf} \text{ holds for }u\text{ for some }(y,s)\in V \}$. Then
 \begin{equation} \label{est_measure_vertices_touching}
    |V|  \leq \frac{1}{\lambda^d} \left( 1+\frac{L}{a} + \Lambda d \right)^{d+1} |W|. 
 \end{equation}
\end{lemma}

\begin{proof} 
The proof is divided into two steps.
 \setcounter{step}{0}
 \vspace{-0.15cm}
\begin{step} \label{step_red_infconvol}
We make two reductions. 
First, by replacing $u$ by $u+\alpha \left(\frac{1}{2}|x|^2 -t \right)$ and $L$ by $L+C\alpha$ 
and letting $\alpha \rightarrow 0$, we may suppose that there exists $\eta>0$ such that
for every $(y,s)\in Q_1$, 
\begin{multline} \label{ramener_inf_intQ1}
 \min_{(z,\tau)\in Q_1 \setminus Q_{1-\eta}, \tau \leq t} \left( u(z,\tau) +\tfrac{a}{2}|z-y|^2 - a \tau  \right)  \\
 > \inf_{(z,\tau)\in Q_{1-\eta}, \tau \leq t} \left( u(z,\tau) +\tfrac{a}{2}|z-y|^2 - a \tau  \right). 
 \end{multline}
Next we make a reduction to the case that $u$ is semiconcave by an infimal convolution approximation. 
According to \eqref{ramener_inf_intQ1}, for every sufficiently small $\eps>0$, 
there exist $0<r_\eps<1$ and $-1<T_\eps<0 $ such that, 
for each $(y,s) \in V$, there exists $(x,t) \in Q_{1-r_\eps}(0,T_\eps)$ such that
 \begin{multline}\label{def_point_inf_eps}
u_\eps(x,t) - \inf_{Q_{1-r_\eps}(0,T_\eps)}u_\eps +\tfrac{a}{2}|x-y|^2 - a(t-s) \\
= \inf_{(z,\tau)\in Q_{1-r_\eps}(0,T_\eps),\tau \leq t}
\left( u(z,\tau)-\inf_{Q_{1-r_\eps}(0,T_\eps)} u_\eps +\tfrac{a}{2}|z-y|^2-a(\tau -s)\right)=0, 
 \end{multline}
 and 
\begin{equation*}
 \lim_{\eps \rightarrow 0} r_\eps =0 \qquad \mbox{and} \qquad  \lim_{\eps \rightarrow 0} T_\eps =0.  
\end{equation*}
Set 
\begin{equation*}
 W_\eps:=\{(x,t) \in Q_{1-r_\eps}(0,T_\eps) : \eqref{def_point_inf_eps} \mbox{ holds for }u_\eps \mbox{ for some }(y,s)\in V \}. 
\end{equation*}
Assume that we have shown that 
\begin{equation} \label{inclusion_W_Weps}
\limsup_{\eps \rightarrow 0} W_\eps \subseteq W, 
\quad \mbox{ where } \quad \limsup_{\eps \rightarrow 0} W_\eps := \bigcap_{\eps>0} \bigcup_{0<\delta \leq \eps} W_{\delta}  , 
\end{equation}
and for all $\eps>0$, 
\begin{equation}\label{est_vert_touching_eps}
   |V| \leq \frac{1}{\lambda^d} \left( 1+\frac{L}{a} + \Lambda d \right)^{d+1} |W_\eps|. 
\end{equation}
Then, since $\sup_{0<\eps<1/2}|W_\eps|\leq |Q_1|<+\infty$, the inequality
$ \limsup_{\eps \rightarrow 0}  \left|W_\eps \right|\leq \left| \limsup_{\eps \rightarrow 0} W_\eps \right|$ holds true 
and we have  
\begin{equation*}
|V| \underset{\eqref{est_vert_touching_eps}}{\leq} 
     \frac{1}{\lambda^d} \left(1+\frac{L}{a}+\Lambda d \right)^{d+1} \left|\limsup_{\eps \rightarrow 0}W_\eps \right| 
    \underset{\eqref{inclusion_W_Weps}}{\leq} 
     \frac{1}{\lambda^d} \left( 1+\frac{L}{a}+ \Lambda d \right)^{d+1} \left|W\right|. 
\end{equation*}
Thus we deduce \eqref{est_measure_vertices_touching}. 

To obtain the lemma, it remains to show the assertions~\eqref{inclusion_W_Weps} and \eqref{est_vert_touching_eps}. 

For $\eqref{inclusion_W_Weps}$, let $(x,t) \in \limsup_{\eps \rightarrow 0} W_\eps$.
Up to a subsequence, we can assume that $(x,t)\in W_\eps \cap Q_{1-r_\eps}(0,T_\eps)$ for all $\eps>0$. 
By the definition of~$W_\eps$, there exists $(y_\eps,s_\eps) \in V$ such that  
\begin{multline} \label{hyp_eps}
 u_\eps(x,t) - \inf_{Q_{1-r_\eps}(0,T_\eps)} u_\eps +\tfrac{a}{2}|x-y_\eps|^2 - a(t-s_\eps)  \\
 = \inf_{(z,\tau)\in Q_{1-r_\eps}(0,T_\eps), \tau \leq t}
 \left( u_\eps(z,\tau) - \inf_{Q_{1-r_\eps}(0,T_\eps)} u_\eps+\tfrac{a}{2}|z-y_\eps |^2 - a (\tau -s_\eps) \right)=0.
\end{multline}
Since $V$ is compact, up to extracting a subsequence, 
there exists $(y,s)\in V$ such that $(y_\eps,s_\eps) \rightarrow (y,s)$ as $\eps \rightarrow 0$. 
By convergence of $u_\eps$, we deduce that $\ds u_\eps(x,t)  \sublim_{\eps \rightarrow 0}  u(x,t)$. 
Moreover, since~$u_\eps \geq u$
and $u_\eps \rightarrow  u$ locally uniformly on $Q_1$, we have
\begin{equation*}
 \lim_{\eps\rightarrow 0 } \inf_{Q_{1-r_\eps}(0,T_\eps)} u_\eps =   \inf_{Q_1} u. 
\end{equation*}
Letting $\eps \rightarrow 0$ in \eqref{hyp_eps} yields $(x,t) \in W$. 
This completes the proof of \eqref{inclusion_W_Weps}, and therefore it remains to prove \eqref{est_vert_touching_eps}, 
that is, the statement of the lemma under the extra assumption that $u$ is semiconcave. 
\end{step} 
\begin{step}
Assuming $u$ is semiconcave, we give the proof of~\eqref{est_measure_vertices_touching}.
Select a Lebesgue-measurable function $\ol Z$ : $V \rightarrow Q_1$ such that 
the map 
\begin{equation*}
(z,t)\mapsto u(z,t) - \inf_{Q_1} u +a\left(\tfrac{1}{2} |z-y|^2 - (t-s) \right) 
\end{equation*}
 attains its infimum in $Q_1$ at $(z,t)=\ol Z(y,s)$ and this infimum is equal to zero. 
 For example, we may take~$\ol Z(y,s)$ to be the lexicographically least element of the (necessarily closed) set of infima.
 The function $u$ is~$C^{1,1}$ on~$A: =\ol Z(V)$ and $\ol Z$ has a Lipschitz inverse $\ol Y=(\ol y,\ol s)$ given by
 \begin{equation*}
\left\{
 \begin{aligned}
 \ol y(z,t) & =z  + \frac{1}{a}Du(z,t)  ,       \\
  \ol s(z,t)& =t  - \frac{1}{a}u(z,t) - \frac{1}{2} |z- \ol y(z,t)|^2.
  \end{aligned}
  \right.
 \end{equation*}
By Rademacher's theorem, $\ol Y$ is differentiable almost everywhere on $A$ for the $d+1$-dimensional Lebesgue measure.  
Then, by using the Lebesgue differentiation theorem, we see that $u$ is twice differentiable in space and differentiable in time 
at almost every point of $(z,t) \in A$ and, at such~$(z,t)$, we have
\begin{equation*}
 D^2u(z,t) \geq - a I_d \qquad \mbox{and} \qquad \partial_tu(z,t) \leq a. 
\end{equation*}
Thus, 
\begin{equation*}
 D \ol y (z,t)= I_d+\frac{1}{a} D^2u(z,t)  \geq 0
\end{equation*}
as well as
\begin{multline*}
  - \lambda \mbox{tr} (D \ol y(z,t)) = \mathcal{P}^+_{\lambda, \Lambda} (D\ol y(z,t)) 
  =\mathcal{P}^+_{\lambda, \Lambda} \left( I_d+\frac{1}{a} D^2u(z,t) \right)  \\
  \geq \frac{1}{a} \mathcal{P}^+_{\lambda, \Lambda} \left( D^2u(z,t) \right)
+ \mathcal{P}^-_{\lambda, \Lambda} \left( I_d \right)\geq  \frac{1}{a}(-L -\partial_tu(z,t)) - \Lambda d, 
\end{multline*}
and therefore
\begin{equation*}
 0\leq D \ol y (z,t) \leq \frac{1}{\lambda} \left( 1+ \frac{L}{a} + \Lambda d \right).
\end{equation*}
Similarly $D^2u (z,t) \geq -aI_d$ implies that
\begin{equation*}
 1-\frac{1}{a}\partial_tu(z,t)  \leq 1-\frac{1}{a}\left(-L -\mathcal{P}^+_{\lambda, \Lambda}(-aI_d) \right) = 1+\frac{L}{a}+\Lambda d.
\end{equation*}
An application of the area formula for Lipschitz functions gives
\begin{multline*}
|V|= \int_A |\det D\ol Y(z,t)|dz dt =\int_A \det \left(I+\frac{1}{a} D^2u(z,t) \right) \left(1-\frac{1}{a} \partial_tu(z,t) \right) dz dt  \\
 \leq \frac{1}{\lambda^d} \left(1+ \frac{L}{a} + \Lambda d \right)^d \left(1+\frac{L}{a}+\Lambda d \right)|A|,
\end{multline*}
from which we obtain the lemma, using that $A\subseteq W$. \qedhere 
\end{step}
\end{proof}

In our analysis, an important role will be played by the functions $\phi$ which we define by
\begin{equation*}
 \phi (x,t): = c (t+\tau)^{-b}  \left( e^{- a \frac{|x|^2}{ t+\tau}} -  e^{ - a \theta^{-1}}\right).
\end{equation*}
The parameters $a,b$ and $c$ will be adjusted with the uniform ellipticity constants and the opening~$\theta$ of the parabolic balls of the form $ G_{\theta,1+\tau}(0,-\tau)$
with $\tau>0$. More precisely, we will consider the choice given by
\begin{equation}\label{choice_a}
a = \max \left\{ \frac{1+d \Lambda \theta}{2\lambda}, \theta \right\}
\end{equation}
and 
\begin{equation}\label{choice_b}
b = \max \left\{ \frac{2d\Lambda a+1}{1-e^{\frac{1}{2}-a\theta^{-1}}} , \frac{4\lambda a-1}{e^{\frac{1}{2}-a\theta^{-1}}}\right\}
\end{equation}
and 
\begin{equation}\label{choice_c}
c= 2(1+\tau)^{b+1} e^{ a \theta^{-1}}.
\end{equation}
We next show that, with this choice of parameters, $\phi$ is a nonnegative subsolution in 
$G_{\theta, 1+\tau}(0, - \tau)$ which vanishes on the lateral boundary of $G_{\theta, 1+\tau}(0, - \tau)$
and is not too large initially. This plays the role of the ``bump function'' from the elliptic case 
\cite[Lemma 4.1]{caff_cabre}.

\begin{lemma}\label{fund_sol_heat_eq_bar}
Let $\tau>0$. For $a,b$ and $c$ given by \eqref{choice_a}--\eqref{choice_c}, the function $\phi$ satisfies
\begin{equation*}
\begin{cases}
 \partial_t \phi  +\mathcal{P}^+_{\lambda, \Lambda}(D^2\phi )\leq -1, &\mbox{on } G_{\theta,1+\tau}(0,-\tau) \cap \{(y,s):s> 0\}, \\
 \phi = 0             , & \mbox{on } \partial_p G_{\theta,1+\tau}(0,-\tau) \cap \{(y,s) : s>0 \}, \\ 
 \phi>0               , & \mbox{on }  G_{\theta,1+\tau}(0,-\tau) \setminus \partial_p G_{\theta,1+\tau}(0,-\tau), \\
 0\leq \phi \leq \beta, & \mbox{on } G_{\theta,1+\tau}(0,-\tau) \cap \{(y,s):s = 0\},  
 \end{cases}
\end{equation*}
with $\beta>0$ given by
\begin{equation*}
 \beta := 2\frac{(1+\tau)^{b+1}}{\tau^b} (e^{a\theta^{-1}} - 1).
\end{equation*}

\end{lemma}

\begin{proof} Let us introduce the variable $ \rho:= |x|^2/(t+\tau)$ and the function $ \psi $ given by  
\begin{equation*}
  \psi (\rho,t) := c(t+\tau)^{-(b+1)} e^{- a \rho} .  
\end{equation*}
First, by inserting the value of $c$ given by \eqref{choice_c}, observe that  
\begin{equation}\label{est_psi_ab}
\psi (\rho,t)\geq 2  \quad \text{on } G_{\theta,1+\tau} (0,-\tau). 
\end{equation}
The two last properties are immediate to check. Thus we focus on the first assertion.
The time derivative of $\phi $ is given by 
\begin{equation*}
  \partial_t \phi (x,t)= \psi (\rho,t) \left( -(1- e^{a ( \rho- \theta^{-1}) }  ) b +a\rho \right).
\end{equation*}
and the Hessian of $\phi $ is given by
\begin{equation*}
 D^2 \phi (x,t) = 2 a\psi (\rho,t)
  \left( \left(- 1 + 2a \rho \right) \frac{x\otimes x}{|x|^2}   - \left(I -\frac{x\otimes x}{|x|^2} \right)   \right) . 
\end{equation*}
and has eigenvalues $2 a\psi (\rho,t) \left(- 1 + 2a \rho \right)$ with multiplicity $1$ 
and $ -2 a\psi (\rho,t)$ with multiplicity $d-1$. Hence
\begin{equation*}
 \mathcal{P}^+_{\lambda, \Lambda} (D^2 \phi (x,t)) = \psi (\rho,t)
 \begin{cases}
2a \left( d  -2a \rho  \right)\Lambda           ,            & \text{if }  \rho \leq (2a)^{-1}, \\
2 a \left(- \lambda \left(-1  +2 a \rho  \right) +(d-1) \Lambda \right), & \text{if } \rho\geq (2a)^{-1}.
\end{cases}
\end{equation*}
We distinguish two cases. If  $0 \leq  \rho \leq (2a)^{-1}$, 
\begin{align*}
\partial_t\phi(x,t)  + \mathcal{P}^+_{\lambda, \Lambda} (D^2\phi(x,t) ) 
& = \psi (\rho,t)   \left( -(1- e^{a ( \rho - \theta^{-1}) }  ) b + a ( \rho + 2( d  - 2a \rho)  \Lambda)    \right) \\
& \underset{\eqref{choice_a}}{\leq} 
\psi (\rho,t) \left( -(1- e^{\frac{1}{2}-a \theta^{-1}}) b + \tfrac{1}{2} + 2 d a \Lambda    \right)  \\
& \underset{\eqref{choice_b}}{\leq}  - \tfrac{1}{2} \psi (\rho,t).
\end{align*}
By using \eqref{est_psi_ab}, we obtain the desired upper bound. Now assume that $ (2a)^{-1} \leq \rho   \leq  \theta^{-1}$,   
\begin{multline*}
\partial_t\phi(x,t) + \mathcal{P}^+_{\lambda, \Lambda} (D^2\phi(x,t)) 
\leq \psi (\rho,t) \left( -(1- e^{a ( \rho - \theta^{-1}) }  ) b + a (1 - 4 a \lambda) \rho + 2 d a \Lambda    \right)\\
\leq \psi (\rho,t) \left( -b + 2 a d \Lambda+ b e^{a ( \rho - \theta^{-1}) } + a (1 - 4 a \lambda) \rho   \right).
\end{multline*}
The function $\rho \mapsto b e^{a (\rho-\theta^{-1})} + a (1-4a \lambda) \rho$ 
is decreasing on $]-\infty, \rho_0]$ and increasing on $[\rho_0, +\infty[ $ with 
$\rho_0:=  \tfrac{1}{a} \ln \left( \tfrac{4 a \lambda - 1}{b} \right) +\theta^{-1}$.
By \eqref{choice_b}, the coefficients $a$ and $b$ are chosen so that $\rho_0 <(2a)^{-1}$. 
Under this assumption, the upper bound on the interval $[(2a)^{-1}, \theta^{-1}]$ corresponds to $\rho= \theta^{-1}$ which provides
\begin{multline*}
\partial_t\phi(x,t)  + \mathcal{P}^+_{\lambda, \Lambda} (D^2\phi (x,t)) 
   \leq \psi (\rho,t) a  \left( 2  d \Lambda   +  \theta^{-1} - 4 \theta^{-1} a \lambda   \right) \\
   \underset{\eqref{choice_a}}{\leq} - \psi (\rho,t)   a  \theta^{-1}  
   \underset{\eqref{choice_a}}{\leq} - \psi (\rho,t) .   
\end{multline*}  
By recalling \eqref{est_psi_ab}, we obtain the desired upper bound.
\end{proof}

The following lemma contains the measure theoretic information necessary to conclude the proof of Proposition~\ref{W_2eps_para}. 
The argument relies on Lemmas~\ref{lemma_abp_measure_inf} and~\ref{fund_sol_heat_eq_bar}.

\begin{lemma} \label{measure_estimate_macroscopic_ball}
Let $3/4 \leq \theta$, $h_0>0$, $\kappa_1>0$, $A_\kappa$ be defined by \eqref{def_A_op} and $G^{-}_{\theta,h_0}(x_0,t_0) \subseteq Q_1$. 
 There exist constants $M \geq 1$ and $\sigma>0$ depending only on $\theta$ and $d$, $\lambda$, $\Lambda$ such that, 
 if $(x_0,t_0) \in A_{\kappa_1}$ and $ \kappa \geq \kappa_1$, then 
 \begin{equation*}
  |A_{M\kappa} \cap G^-_{\theta,h_0}(x_0,t_0)| \geq |G^-_{\theta,h_0}(x_0,t_0) \cap A_{\kappa}|
  +\sigma \eta |G^-_{\theta,h_0}(x_0,t_0) \setminus A_{\kappa} |.
 \end{equation*}
\end{lemma}
\begin{proof}
We decompose the measure estimate into two parts
 \begin{equation*}
 |A_{M\kappa} \cap  G^{-}_{\theta, h_0}(x_0,t_0)|
 = | A_{\kappa} \cap  G^{-}_{\theta, h_0}(x_0,t_0)| + |(A_{M\kappa} \setminus A_{\kappa}) \cap  G^{-}_{\theta, h_0}(x_0,t_0)|.
\end{equation*}
It is enough to estimate $|(A_{M\kappa} \setminus A_{\kappa}) \cap G^{-}_{\theta, h_0}(x_0,t_0)|$. We claim that
\begin{equation*}
\left| G^{-}_{\theta, h_0}(x_0,t_0) \setminus A_\kappa \right| 
\leq \frac{1}{\sigma \eta} \left| \left( A_{M\kappa}\setminus  A_{\kappa} \right) \cap G^{-}_{\theta, h_0}(x_0,t_0)  \right|.   
\end{equation*}
For $\kappa \geq \kappa_1$, we define the collection of parabolic balls given by 
\begin{multline*}
 \mathcal{B}:= \bigg\{ G_{\theta,h} (x,t) : (x,t) \in  G^{-}_{\theta, h_0}(x_0,t_0), \\
   G_{\theta,h} (x,t) \cap  G^{-}_{\theta, h_0}(x_0,t_0) \cap   \left\{ (y,s): s<t+h \right\}
   \subseteq Q_1\setminus A_\kappa  \\
  \text{and }  G_{\theta,h}(x,t) \cap \{(y, s) : s= t+h\} \cap  G^{-}_{\theta, h_0}(x_0,t_0)
  \cap A_\kappa \neq \emptyset  \bigg\}.
\end{multline*}
Notice that for all $(x,t) \in G^{-}_{\theta, h_0}(x_0,t_0)$, the point $(x_0,t_0)$ belongs to 
the parabolic ball $G_{\theta,t_0-t}(x,t)$. 
Observe that $(x_0,t_0)\in A_{\kappa}$ by applying Lemma~\ref{comparison_Aop_Theta} with $(x_0,t_0)\in A_{\kappa_1}$.
This implies that for all $ G_{\theta,h} (x,t)\in \mathcal{B}$, $h\leq t_0-t \leq h_0$.
Then, by Lemma~\ref{vitali_par_balls}, we may extract from  $\mathcal{B}$ a countable subcollection
 $\{ G^{}_{\theta,h_{i}}(x_i,t_i) : i \in \N \}$ such that the  $G^{}_{\theta,h_{i}}(x_i,t_i)$ are disjoint,  
  \begin{equation*} 
  G^{-}_{\theta, h_0}(x_0,t_0) \setminus A_\kappa  \subseteq \bigcup_{i\in \N}  \widehat  G^{}_{\theta,h_{i}}(x_i,t_i)
 \quad \text{ and } \quad \frac{| G^{}_{\theta,h_{i}}(x_i,t_i)|}{|\widehat{G}^{}_{\theta,h_{i}}(x_i,t_i)|} = \eta, 
 \end{equation*}
 with $\eta$ given by \eqref{ratio_pb_vitali}. By combining these, we get 
\begin{equation*}
 |  G^{-}_{\theta, h_0}(x_0,t_0) \setminus A_\kappa|
\leq \sum_{i\in \N} |\widehat G_{\theta,h_i}(x_i,t_i)| = \frac{1}{\eta} \sum_{i\in \N} |G_{\theta,h_i}(x_i,t_i)|.
  \end{equation*}
Next we complete the proof under the assumption that for all $i \in \N$, 
\begin{equation} \label{est_measure_ball_vitali}
 |G_{\theta,h_i}(x_i,t_i) \cap G^-_{\theta,h_0}(x_0,t_0) \cap A_{M\kappa}| \geq \sigma |G_{\theta,h_i}(x_i,t_i)|
\end{equation}
for some constants $M>1$ and $ \sigma>0$, depending only on $\theta$, $d$, $\lambda$ and $\Lambda$.
Using also that the selected balls are disjoint, we obtain that 
\begin{multline*}
\sum_{i\in \N} |G_{\theta,h_i}(x_i,t_i)| 
\leq \frac{1}{\sigma} \sum_{i\in \N}   \left| G_{\theta,h_i}(x_i,t_i) \cap G^-_{\theta,h_0}(x_0,t_0) \cap A_{M\kappa}  \right|  \\
= \frac{1}{\sigma}    \left|\bigcup_{i\in \N} G_{\theta,h_i}(x_i,t_i) 
\cap \left\{ (y,s): s<t_i+h_i \right\} \cap G^-_{\theta,h_0}(x_0,t_0) \cap A_{M\kappa}  \right|.  
\end{multline*}
Since every ball in $\mathcal{B}$ satisfies 
$G_{\theta,h}(x,t) \cap   \left\{ (y,s): s<t+h \right\} \cap G^-_{\theta,h_0}(x_0,t_0) 
\subseteq G^{-}_{\theta, h_0}(x_0,t_0) \setminus A_{\kappa}$, 
we deduce that 
\begin{equation*}
| G^{-}_{\theta, h_0}(x_0,t_0) \setminus A_\kappa |
\leq \frac{1}{\sigma \eta} \left|G^{-}_{\theta, h_0}(x_0,t_0) \cap \left( A_{M\kappa}\setminus  A_{\kappa} \right) \right|. 
\end{equation*}
The proof is complete, pending the verification of~\eqref{est_measure_ball_vitali}, which is achieved in the next lemma.
\end{proof}

The following lemma is the key step in the proof of Proposition~\ref{W_2eps_para}. 

\begin{lemma} \label{est_measure_small_pb_bd} 
Let $3/4 \leq \theta$, $\kappa_1> 0$  and  $G^-_{\theta,h_0}(x_0,t_0)  \subseteq Q_1$. Suppose $u\in \lsc(Q_1)$ satisfies 
\begin{equation}\label{u_supersol_pucci}
 \partial_tu +\mathcal{P}^+_{\lambda,\Lambda} (D^2u) \geq 0.
\end{equation}
 There exist constants $M \geq 1$ and $\sigma>0$ depending only on $\theta$ and $d$, $\lambda$, $\Lambda$   
such that if $(x_0,t_0) \in A_\kappa$, $\kappa \geq \kappa_1$, then
for all $(x,t) \in G^-_{\theta,h_0}(x_0,t_0)$ satisfying
\begin{equation*}
 G_{\theta,h}(x,t) \cap \{(y,s): s=t+h\} \cap G^-_{\theta,h_0}(x_0,t_0) \cap A_\kappa \neq \emptyset,
\end{equation*} 
we have
\begin{equation*}
|G_{\theta,h}(x,t) \cap A_{M\kappa} \cap G^-_{\theta,h_0}(x_0,t_0)| \geq \sigma |G_{\theta,h}(x,t)|. 
\end{equation*}
\end{lemma}
\begin{proof}
 Let $(z_1,t_1)\in G_{\theta,h}(x,t) \cap \{(y,s): s=t+h\} \cap G^-_{\theta,h_0}(x_0,t_0) \cap A_\kappa$. 
 By the definition of $A_\kappa$, we can touch $u$ at $(z_1,t_1)$ from below by a concave paraboloid of opening $\kappa$: 
 there exist $y_1 \in B_1$  and $ - 1 < s_1 \leq t_1$ such that
  \begin{multline}  \label{cond_inf_z1_tau1} 
   u(z_1,t_1) - \inf_{Q_1} u  +\tfrac{\kappa}{2} |z_1-y_1|^2 - \kappa (t_1-s_1) \\
 =\inf_{(z,\tau) \in Q_1, \tau \leq t_1} \left( u(z,\tau) -\inf_{Q_1}u +\tfrac{\kappa}{2} |z-y_1|^2 -\kappa (\tau -s_1) \right) =0.
 \end{multline}  
 \setcounter{step}{0}
\begin{step} \label{stepm0}
 Let  $Q_r(x_2,t_2)\subseteq G_{\theta,h}(x,t) \cap G^-_{\theta,h_0}(x_0,t_0)$ be the cylinder given by Lemma~\ref{construct_cylinder_pb} 
 (see Figure~\ref{cube_Qrx2t2}). 
 In particular, 
 \begin{equation} \label{est_petit_cyl_r}
   t_2= t+\frac{h}{2} \quad \text{and} \quad r\geq \frac{1}{\nu} \sqrt{\frac{h}{\theta}} \quad \text{with } \nu =  \frac{4}{\sqrt{2}-1}. 
 \end{equation}
 If we set 
 \begin{equation} \label{est_alpha_delta}
      \alpha : = \min \left\{1 , \frac{ \theta}{(1+\sqrt{2})^2} \right\}  
\quad \text{ and } \quad
  \delta : = \frac{1}{16\nu^2} \frac{\alpha}{\theta}>0, 
 \end{equation}
we claim that the parabolic ball  (see Figure~\ref{cube_Qrx2t2}) $$G_2 : = G_{\alpha, h \left(\frac12+\delta \right)} (x_2,t_2 - \delta h)$$ satisfies 
 \begin{equation} \label{geometric_facts1}
  G_2 \cap \{(y,s) : s \leq t_2\} \subseteq \ol Q_{r/4} (x_2,t_2)
  \end{equation}
  and
 \begin{equation} \label{geometric_facts2}
  (z_1,t_1) \in  (G_2 \setminus \partial_p G_2) \cap \{(y,s) : s=t_1\}.
 \end{equation}
To obtain the first assertion, it suffices to show that the vertex $(x_2,t_2 - \delta h)$ of $G_2$
is in $ \ol Q_{r/4} (x_2,t_2)$ and $G_2 \cap \{(y,s): s= t_2 \} \subseteq \ol B_{r/4} (x_2)$. 
First,  by using \eqref{est_petit_cyl_r} and~\eqref{est_alpha_delta}, the inequality
\begin{equation*}
 \delta h \leq \left(\frac{1}{16\nu^2} \frac{\alpha}{\theta}\right)
 \theta \nu^2 r^2 = \frac{\alpha}{16} r^2 \leq  \frac{r^2}{16} 
\end{equation*}
implies $(x_2,t_2 - \delta h)\in \ol Q_{r/4} (x_2,t_2)$.
By the definition of the parabolic ball $G_2$ and using \eqref{est_petit_cyl_r} and~\eqref{est_alpha_delta}, 
each $(z,t_2) \in  G_2$ satisfies $ |z-x_2|^2 \leq \alpha^{-1}\delta h\leq \frac{r^2}{16}$.  
Then, for the second assertion, observing that $(z_1,t_1)\in G_{\theta,h}(x,t)  \cap \{(y,s): s=t+h\}$, 
it is enough to show that 
\begin{equation*}
G_{\theta,h}(x,t)  \cap \{(y,s): s=t+h\} \subseteq (G_2 \setminus \partial_p G_2 )  \cap \left\{(y,s) : s=t_2+\frac{h}{2} \right\} .
\end{equation*}
Let $(z,\tau) \in G_{\theta,h}(x,t) \cap \{(y,s): s=t+h\}$.
Since $(x_2,t_2) \in G_{\theta,h}(x,t) \cap \{(y,s): s=t+h/2\}$  and $\delta>0$, we get
 \begin{equation*}
  |z-x_2| \leq |z-x| + |x-x_2| 
  = \left(1+ \frac{1}{\sqrt{2}} \right) \sqrt{\frac{h}{\theta}} \underset{\eqref{est_alpha_delta}}{<} \sqrt{\left(\frac 12 + \delta \right) \frac{h}{\alpha}}.
 \end{equation*}
 This is equivalent to $(z,\tau) \in (G_2 \setminus \partial_p G_2) \cap \left\{(y,s) : s=t_2+\frac{h}{2} \right\}$, 
 which gives \eqref{geometric_facts2}.
\end{step}
\begin{figure}[t]
\centering
 \scalebox{0.88}{\includegraphics{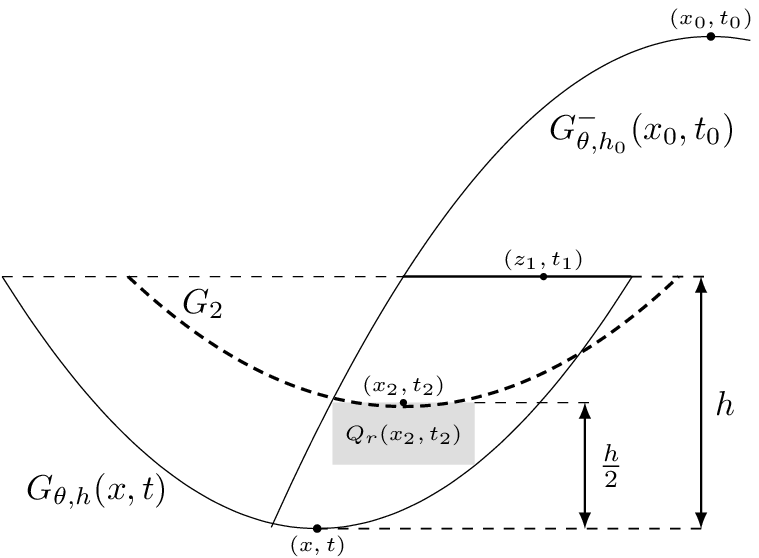}}
 \scalebox{0.88}{\includegraphics{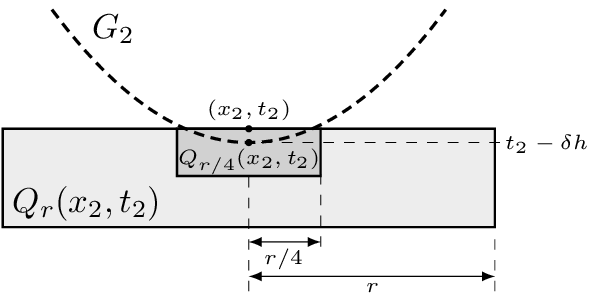}}
\caption{
$Q_r(x_2,t_2) \subseteq G_{\theta,h}(x,t)  \cap G^-_{\theta,h_0}(x_0,t_0)$ and $G_2 = G_{\alpha, h(1/2+\delta)} (x_2,t_2 - \delta h)$.} 
 \label{cube_Qrx2t2} 
\end{figure}

\begin{step} \label{stepm1}
 We claim that there exists $(z_2,t_2)\in G_2 \cap \{(y,s) : s = t_2 \}$ such that
 \begin{equation}\label{est_z2t2_barrier}
  u(z_2,t_2) - \inf_{Q_1} u + \tfrac{\kappa}{2} |z_2-y_1|^2 - \kappa ( t_2  - s_1)  \leq  (d\Lambda+3) \kappa \beta \theta \nu^2 r^2.
 \end{equation}
By applying (a properly scaled) Lemma~\ref{fund_sol_heat_eq_bar}, there exists a barrier function $w$ which satisfies
\begin{equation}\label{barrier_rescaled}
\begin{cases}
 \partial_t w +\mathcal{P}^+_{\lambda,\Lambda}(D^2w)\leq -1, & \text{in }G_2 \cap \{(y,s) : s > t_2 \},  \\
  w=0 ,                        & \text{on }\partial_p G_2 \cap \{(y,s) : s>t_2 \}, \\
  0\leq w \leq \beta h,        & \text{on }G_2 \cap \{(y,s) : s =t_2\} ,  
\end{cases}
\end{equation}
and $w>0$ in $ (G_2\setminus \partial_p G_2) \cap \{(y,s) : s>t_2 \}$. 
In particular, this implies by Step~\ref{stepm0} that $w(z_1,t_1)>0$. 
We have that $w\leq \beta h$ in $G_2 \cap \{ (y,s) : s >t_2 \}$ by the maximum principle. Observe that the function 
\begin{equation*}
 \vphi(z,\tau):= ((d\Lambda+1) \kappa +2 \kappa_1) w -  \tfrac{\kappa}{2} |z-y_1|^2 + \kappa \tau, 
\end{equation*}
satisfies 
\begin{equation*}
\partial_t \vphi +\mathcal{P}^+_{\lambda,\Lambda}(D^2\vphi) \leq -2 \kappa_1. 
\end{equation*}
Notice that $u$ satisfies 
\begin{equation} \label{eq_pucci_u_op_ajuste}
\partial_t u +\mathcal{P}^+_{\lambda,\Lambda}(D^2u) \geq  - \kappa_1. 
\end{equation}
The comparison principle implies that the map $(z,\tau) \mapsto u(z,\tau) - \vphi(z,\tau)$ attains its infimum in 
$G_2\cap \{(y,s) : s >t_2\}$ at some point $(z,\tau)=(z_2,t_2) \in \partial_p (G_2 \cap \{(y,s) : s >t_2\})$.
It is impossible that $(z_2,t_2) \in \partial_p G_2 \cap \{(y,s) : s >t_2 \}$ since \eqref{cond_inf_z1_tau1}, 
$w=0$ on $\partial_p G_2 \cap \{ (y,s) : s>t_2 \}$ and $w(z_1,t_1)>0$ imply that
\begin{align*}
 u(z_1,t_1) - & \vphi(z_1,t_1) =  u(z_1,t_1)  + \tfrac{\kappa}{2} |z_1-y_1|^2 - \kappa t_1
  - ((d\Lambda+1) \kappa +2 \kappa_1) w(z_1,t_1) \\
& \underset{\eqref{cond_inf_z1_tau1}}{<}
\inf_{\substack{(z,\tau) \in Q_1 \\ \tau \leq \tau_1}} \left( u(z,\tau) +\tfrac{\kappa}{2} |z-y_1|^2- \kappa \tau \right)  \\
& \leq \inf_{\substack{(z,\tau) \in \partial_p G_2 \\ t_2 < \tau \leq t_1}} \left( u(z,\tau) +\tfrac{\kappa}{2} |z-y_1|^2- \kappa \tau \right)
     = \inf_{\substack{(z,\tau) \in \partial_p G_2 \\ t_2 < \tau \leq t_1}} \left( u(z,\tau)  - \vphi(z,\tau) \right).
\end{align*}
Moreover, it is impossible that $(z_2,t_2)$ satisfies $t_2=t_1$ since $\vphi$ satisfies
\begin{equation*}
 \partial_t \vphi(\cdot,t_1) +\mathcal{P}^+_{\lambda,\Lambda}(D^2\vphi (\cdot,t_1))\leq -2 \kappa_1 
 \quad \text{for } G_2 \cap \{ (y,s) : s=t_1\}.
\end{equation*}
Hence $(z_2,t_2)\in G_2 \cap \{(y,s): s = t_2 \}$ and so, in particular, by \eqref{geometric_facts1} satisfies $|z_2-x_2| \leq \frac{1}{4} r$,
and
\begin{equation*}
 \vphi(z_2,t_2) = ((d\Lambda+1) \kappa +2 \kappa_1) w(z_2,t_2) -  \tfrac{\kappa}{2} |z_2-y_1|^2 + \kappa t_2. 
\end{equation*}
Using that $w>0$  in $G_2 \cap \{ (y,s) : s >t_2 \}$, we obtain that
\begin{multline*}
u(z_1,t_1)+ \tfrac{\kappa}{2} |z_1-y_1|^2 - \kappa t_1  \geq u(z_1,t_1) -\vphi(z_1,t_1) \\
\geq  \inf_{G_2 \cap \{s>t_2 \}} \left( u(z,\tau) - \vphi(z,\tau) \right) 
  = u(z_2,t_2) -\vphi(z_2,t_2) \\=  u(z_2,t_2) + \tfrac{\kappa}{2} |z_2-y_1|^2 - \kappa t_2 - ((d\Lambda+1) \kappa +2 \kappa_1) w(z_2,t_2). 
\end{multline*}
By combining \eqref{est_petit_cyl_r} and \eqref{barrier_rescaled}, we know that
\begin{equation*}
 w(z_2,t_2) \leq  \beta \theta \nu^2 r^2. 
 \end{equation*}
Using this together with \eqref{cond_inf_z1_tau1}, 
$(z_1,t_1) \in (G_2 \setminus \partial_p G_2) \cap \{(y,s): s=t_2+\frac{h}{2}\}$ (by~\eqref{geometric_facts2}) and~$\kappa \geq \kappa_1$,
\begin{multline*}
 \inf_{(z,\tau) \in Q_1,  \tau \leq t_1}  \left( u(z,\tau) +\tfrac{\kappa}{2} |z-y_1|^2 - \kappa \tau  \right) \\
  \geq u(z_2,t_2) + \tfrac{\kappa}{2} |z_2-y_1|^2 - \kappa t_2 - \left(d\Lambda+3 \right) \kappa \beta \theta \nu^2 r^2.
\end{multline*}
Recalling \eqref{cond_inf_z1_tau1}, we obtain \eqref{est_z2t2_barrier}.
\end{step}

\begin{step}\label{stepm2}
Let $\gamma:=17 (d\Lambda+3) \beta \theta \nu^2$. 
We claim that for all $(y_2,s_2)\in G^-_{1/2, r^2/16} \left(z_2,t_2 - \frac{r^2}{16}\right)$, 
there exists a point $(z_3,t_3)\in G_{1/2,t_2-s_2}(y_2,s_2)\subseteq Q_r(x_2,t_2)$ such that
\begin{equation} \label{good_point_in_the_good_set}
\chi(z_3,t_3) = \inf_{\substack{(z,\tau) \in Q_1\\ \tau \leq t_3 }} \chi(z,\tau)=0, 
\end{equation}
where the function $\chi$ is given by 
\begin{equation*}
\chi(z,\tau): =u(z,\tau) -\inf_{Q_1}u + \tfrac{\kappa}{2} |z-y_1|^2 - \kappa (\tau-s_1)+\tfrac{\gamma \kappa}{2} |z-y_2|^2 -\gamma  \kappa (\tau -s_2) .
\end{equation*}
Observe the function $\chi$ can be written in the form
\begin{equation*}
 \chi(z,\tau) =u(z,\tau) -\inf_{Q_1}u - P_{y_1,s_1;\kappa} (z,\tau) - P_{y_2,s_2; \gamma \kappa} (z,\tau) 
\end{equation*}
where $P_{y_1,s_1;\kappa}$ and $P_{y_2,s_2; \gamma \kappa} $ are the concave paraboloids 
\begin{equation*}
\left\{
\begin{aligned}
& P_{y_1,s_1;\kappa} (z,\tau):= - \tfrac{\kappa}{2} |z-y_1|^2 + \kappa (\tau-s_1)\\
& P_{y_2,s_2; \gamma \kappa} (z,\tau) := - \tfrac{\gamma \kappa}{2} |z-y_2|^2 + \gamma  \kappa (\tau -s_2).
\end{aligned}
\right.
\end{equation*}
 Let $(y_2,s_2)\in G^-_{1/2, r^2/16}\left(z_2,t_2 - \frac{r^2}{16}\right)$. By combining Steps~\ref{stepm0}--\ref{stepm1}, observe that 
\begin{equation*}
(z_2,t_2) \in G_2 \cap \{(y,s) : s=t_2\} \subseteq \ol Q_{r/4}(x_2,t_2). 
\end{equation*}
 By applying Lemma~\ref{construct_cylinder_pb} part~\ref{inclusion_ball_contact_set}, we get 
 $G_{1/2,t_2-s_2}(y_2,s_2)\subseteq Q_r(x_2,t_2)$.  
 Define 
 \begin{equation*}
   \zeta(t):=  \inf \left\{\chi(z,\tau): (z,\tau)\in  G_{1/2,t_2-s_2}(y_2,s_2) ,  \tau \leq t  \right\}.
 \end{equation*}
 To prove the claim, we will use the three following facts:
   \begin{multline}\label{eval_infbigg0}
 t \mapsto \inf \left\{\chi(z,\tau): (z,\tau)\in Q_1 \setminus  G_{1/2,t_2-s_2}(y_2,s_2), \tau \leq t \right\}  \\
 \text{ is nonnegative on }  [s_2,t_2], 
\end{multline}
and
\begin{equation}\label{eval_chi_z2t2}
 \chi(z_2,t_2)<0 \leq  \chi(y_2,s_2), 
\end{equation}
and  
\begin{multline}\label{cont_inf}
 t\mapsto \zeta(t) \text{ is non-increasing and right-continuous on } [s_2,t_2], \\
 \text{ and left-continuous on } [s_2,t_2] \cap \{s: \zeta(s)< 0 \}.  
\end{multline}
Assuming that we have shown \eqref{eval_infbigg0}, \eqref{eval_chi_z2t2} and \eqref{cont_inf}, 
let us prove the claim given by \eqref{good_point_in_the_good_set}.
 Then for all  $t\in [s_2,t_2]$, 
 \begin{multline*}
  \inf \left\{\chi(z,\tau): (z,\tau)\in Q_1 , \tau \leq t  \right\} \\
 = \min \left(\zeta(t), \inf \left\{\chi(z,\tau): (z,\tau)\in Q_1 \setminus  G_{1/2,t_2-s_2}(y_2,s_2),  \tau \leq t \right\} \right). 
 \end{multline*}
 By \eqref{eval_infbigg0}, the second infimum in the right-hand side above is nonnegative  for $t\in [s_2,t_2]$  and it suffices to study the sign of $\zeta(t)$. 
 First notice  $\zeta(s_2)= \chi(y_2,s_2)\geq 0$ by~\eqref{eval_chi_z2t2}. Then, since $(z_2,t_2)\in G_{1/2,t_2-s_2}(y_2,s_2)$ we obtain also by~\eqref{eval_chi_z2t2}  that
\begin{equation*}
  \zeta(t_2) \leq \chi(z_2,t_2)<0.
\end{equation*}
By \eqref{cont_inf}, we deduce there exists a time $t_3\in [s_2,t_2)$ such that
\begin{equation*}
\zeta(t_3)= \inf \left\{ \chi(z,\tau) : (z,\tau)\in   G_{1/2,t_2-s_2}(y_2,s_2)\text{ and } \tau \leq t_3 \right\}=0.   
\end{equation*}
Since $\chi \in \lsc(Q_1)$, there exists $(z_3,t_3) \in G_{1/2,t_2-s_2}(y_2,s_2)$ realizing this infimum, 
and so, satisfying~\eqref{good_point_in_the_good_set}. 
Then, the inclusion $G_{1/2,t_2-s_2}(y_2,s_2)\subseteq Q_r(x_2,t_2)$ yields the claim. 
To complete the proof of \eqref{good_point_in_the_good_set}, it remains to check \eqref{eval_infbigg0}, \eqref{eval_chi_z2t2} 
and \eqref{cont_inf}. 
To get \eqref{eval_infbigg0}, observe that, for all $(z,\tau) \in Q_1 \setminus  G_{1/2,t_2-s_2}(y_2,s_2)$, $\tau \leq t_2$, 
\begin{equation*}
 \tau - s_2 - \tfrac{|z-y_2|^2}{2} <0, 
\end{equation*}
and so, 
\begin{equation*}
\chi(z,\tau) > u(z,\tau) - \inf_{Q_1} u - P_{y_1,s_1;\kappa} (z,\tau). 
 \end{equation*}
By \eqref{cond_inf_z1_tau1}, the right-hand side in the inequality above is nonnegative.
Therefore, we have $\chi(z,\tau) > 0$ for all $(z,\tau)\in Q_1 \setminus G_{1/2,t_2-s_2}(y_2,s_2)$, $\tau \leq t_2\leq t_1$. 
By passing to the inf, we obtain~\eqref{eval_infbigg0}. By repeating \eqref{cond_inf_z1_tau1}, we also obtain that
\begin{equation*}
 \chi(y_2,s_2) =u(y_2,s_2) - \inf_{Q_1} u - P_{y_1,s_1;\kappa} (y_2,s_2) \geq 0.
\end{equation*}
For~\eqref{eval_chi_z2t2}, since $(y_2,s_2)\in G^-_{1/2, r^2/16}\left(z_2,t_2 - \frac{r^2}{16}\right)$, we have
\begin{equation*}
t_2-s_2 - \tfrac{|y_2-z_2|^2}{2} \geq \tfrac{r^2}{16}.
\end{equation*}
This implies  $P_{y_2,s_2; \gamma \kappa} (z_2,t_2)\geq \gamma \kappa \tfrac{r^2}{16}$. By using also Step~\ref{stepm1}, we obtain
\begin{equation*}
\chi(z_2,t_2)  <(d\Lambda+3) \kappa \beta \theta \nu^2 r^2 - \gamma \kappa \tfrac{r^2}{16}.
\end{equation*}
By inserting the value of $\gamma$, we get
\begin{equation*}
(d\Lambda+3) \kappa \beta \theta \nu^2 r^2 - \gamma \kappa \tfrac{r^2}{16}
= \tfrac{\kappa  r^2}{16} \left(16 (d\Lambda+3)\beta \theta \nu^2 - \gamma \right)   <0, 
\end{equation*}
and the claim \eqref{eval_chi_z2t2} follows. 

For \eqref{cont_inf}, it is clear that $\zeta$ is non-increasing. Moreover, $\zeta \in \lsc(Q_1)$ since $u \in \lsc(Q_1)$.
The lower semicontinuity and the monotonicity of $\zeta$ imply that $\zeta$ is right-continuous. 
To show that $\zeta$ is left-continuous, we argue by contradiction.
Assume there exist~$\ol t \in (s_2,t_2]$  and a  strictly increasing sequence $r_k \rightarrow \ol t$ such that
\begin{equation}\label{sequence_points_chi_t}
\zeta(r_k) \sublim_{k\rightarrow +\infty} \zeta^-(\ol t)> \zeta(\ol t), \quad \text{with } \zeta(\ol t) < 0,   
\end{equation}
where $\zeta^-(\tau)$ denotes the limit from the left of $\zeta$ at $\tau$. Define  
\begin{equation*}
 P(z,\tau):= P_{y_1,s_1;\kappa} (z,\tau) + P_{y_2,s_2; \gamma \kappa} (z,\tau).
\end{equation*}
By \eqref{sequence_points_chi_t}, we deduce that for all $(z,\tau)\in G_{1/2,t_2-s_2}(y_2,s_2),\tau <\ol t$, 
\begin{equation}\label{ineq_strict_less_tbar}
 u(z,\tau) - \inf_{Q_1} u  - P(z,\tau)\geq \zeta^-(\ol t), 
\end{equation}
and there exists $(\ol z,\ol t)\in G_{1/2,t_2-s_2}(y_2,s_2)$ such that 
\begin{multline}\label{u_chi_ztol}
\zeta(\ol t) = u(\ol z,\ol t) - \inf_{Q_1} u  - P(\ol z,\ol t) = \chi(\ol z,\ol t) \\
= \inf \left\{ \chi(z,\tau) : (z,\tau)\in G_{1/2,t_2-s_2}(y_2,s_2),\tau \leq \ol t \right\}.
\end{multline}
Notice that by using \eqref{cond_inf_z1_tau1},
$\zeta(\ol t)< 0$ implies that $(\ol z,\ol t) \in G_{1/2,t_2-s_2}(y_2,s_2) \setminus \partial_p G_{1/2,t_2-s_2}(y_2,s_2)$. 
Let $\phi$ be a smooth function such that $u-\phi$ has a local minimum at $(\ol z,\ol t)$. 
Denote $\tilde \phi$ the map  
$\tilde \phi(y,s) :=\phi(y,s) - L(s-\ol t)$ with $L\geq 0$ to be selected below.
For $s\geq \ol t$, using that $\tilde \phi (\ol z,\ol t)=\phi(\ol z,\ol t)$, we obtain
\begin{equation}\label{min_local_positive_time}
 u(y,s)-\tilde \phi(y,s)= (u-\phi) (y,s)+ L(s-\ol t)  \geq (u-\tilde \phi) (\ol z,\ol t). 
 \end{equation}
Consider now $s<\ol t$,
\begin{multline*}
 u(y,s)-  \tilde \phi(y,s) \underset{\eqref{ineq_strict_less_tbar}}{\geq} \zeta^-(\ol t) +\inf_{Q_1} u +P(y,s) - \phi(y,s) + L(s-\ol t) \\
 \underset{\eqref{u_chi_ztol}}{\geq} u(\ol z,\ol t)- \tilde \phi(\ol z,\ol t) 
 +P(y,s)-P(\ol z,\ol t) -(\phi(y,s)- \phi(\ol z,\ol t))  + (\zeta^-(\ol t)-\zeta (\ol t))+ L(s-\ol t). 
\end{multline*}
On the set
\begin{multline*}
 \bigg\{(y,s) \in G_{1/2,t_2-s_2}(y_2,s_2) \setminus \partial_p G_{1/2,t_2-s_2}(y_2,s_2): s\leq \ol t \, \text{ and }  \\
  \max \left\{|P(y,s)-P(\ol z,\ol t)|, |\phi(y,s)- \phi(\ol z,\ol t)|, L|s-\ol t| \right\}
 \leq  \frac{1}{4}(\zeta^-(\ol t)-\zeta (\ol t))  \bigg\}, 
\end{multline*}
the following inequality holds true
\begin{equation}\label{min_loc_strict_dep_L}
 u(y,s)-\tilde \phi(y,s)  > u(\ol z,\ol t) - \tilde \phi(\ol z,\ol t)    +\frac{1}{4} (\zeta^-(\ol t)-\zeta (\ol t)), 
\end{equation}
and by putting together the two cases \eqref{min_local_positive_time} and \eqref{min_loc_strict_dep_L}, 
the map $(y,s) \mapsto (u-\tilde \phi)(y,s)$ has a local minimum at $(\ol z,\ol t)$. 
Using that $u$ is a supersolution of \eqref{u_supersol_pucci}, we obtain
\begin{equation*}
0\leq \partial_t \tilde \phi(\ol z,\ol t) + \mathcal{P}^+_{\lambda,\Lambda} (D^2\tilde \phi(\ol z,\ol t)) 
= - L + \partial_t \phi(\ol z,\ol t) + \mathcal{P}^+_{\lambda,\Lambda} (D^2 \phi(\ol z,\ol t)) .
\end{equation*}
Taking $L:=1+ \partial_t \phi (\ol z,\ol t) + \mathcal{P}^+_{\lambda,\Lambda} (D^2 \phi (\ol z,\ol t)) \geq 1$, we get a contradiction.
\end{step}

\begin{step} \label{stepm3} 
 Consider the function 
 $\ol Z : G^-_{1/2, r^2/16}\left(z_2,t_2 - \tfrac{r^2}{16}\right) \rightarrow Q_1$ given by $\ol Z(y,s)=(\ol z(y,s), \ol \tau(y,s))$, where  
 \begin{equation*}
  \left\{   
     \begin{aligned}
       &\ol z(y,s) := \tfrac{1}{\gamma +1} (y_1+ \gamma y),   \\
       &\ol \tau(y,s) := \tfrac{1}{\gamma +1} \left(s_1+ \gamma s -\tfrac{\gamma}{(\gamma+1)^2} |y_1-y|^2 \right).
     \end{aligned}
  \right.
\end{equation*}
 Observe by completing the square that we have  for all $z\in \R^d, \tau \in \R $,   
 \begin{multline*}
  \tfrac{\kappa}{2} |z-y_1|^2 - \kappa (\tau -s_1)+ \tfrac{\gamma \kappa}{2} |z-y_2|^2 - \gamma \kappa (\tau - s_2)   \\
  = \tfrac{(\gamma+1)\kappa}{2} |z-\ol z(y_2,s_2)|^2 - (\gamma +1)\kappa (\tau- \ol \tau(y_2,s_2)). 
 \end{multline*}
 It follows by Step~\ref{stepm2} that the map 
 \begin{equation*}
  (z,\tau) \mapsto u(z,\tau) - \inf_{Q_1}u +\tfrac{(\gamma +1)}{2}|z-\ol z(y_2,s_2)|^2 -  (\gamma +1) (\tau - \ol \tau(y_2,s_2))
 \end{equation*}
 attains its infimum in $Q_1$ at some point of $Q_r(x_2,t_2)$ and this infimum is equal to zero. 
Since $u$ satisfies \eqref{eq_pucci_u_op_ajuste}, we can apply Lemma \ref{lemma_abp_measure_inf} by taking $L=\kappa_1>0$ and we obtain 
 \begin{multline}\label{touch_vert_barZ}
  \left| \ol Z \left( G^-_{1/2, r^2/16}\left(z_2,t_2 - \tfrac{r^2}{16}\right) \right)\right| \\
    \leq \frac{1}{\lambda^d} \left(1+\frac{\kappa_1}{(\gamma +1) \kappa}+\Lambda d\right)^{d+1}|Q_{r}(x_2,t_2) \cap A_{(\gamma+1)\kappa}|.
 \end{multline}
 Since $\gamma \geq 1$ and $\gamma/(\gamma +1) \geq \tfrac{1}{2}$, we deduce by the change of variables formula that 
 \begin{equation}\label{change_variables_formula}
   \left| \ol Z \left( G^-_{1/2, r^2/16}\left(z_2,t_2 - \tfrac{r^2}{16}\right) \right)\right| 
   \geq 2^{-d-1}  \left| G^-_{1/2, r^2/16}\left(z_2,t_2 - \tfrac{r^2}{16}\right) \right|.  
 \end{equation}
By combining the explicit expression of a parabolic ball and \eqref{borne_inf_r_h},
we deduce there exists a numerical constant $0<c<1$ such that
\begin{equation}\label{est_vol_parabolique}
 \left| G^-_{1/2, r^2/16}\left(z_2,t_2 - \tfrac{r^2}{16}\right) \right|\geq \frac{c}{\theta} |G_{\theta, h}(x,t)|. 
\end{equation}
By Lemma~\ref{construct_cylinder_pb} part \ref{inclusion_Kr_pb}, 
$Q_{r}(x_2,t_2) \subseteq G_{\theta, h}(x,t) \cap G^-_{\theta, h_0}(x_0,t_0)$. 
By combining this observation with \eqref{touch_vert_barZ}, \eqref{change_variables_formula} and \eqref{est_vol_parabolique},
recalling that $\kappa \geq \kappa_1$ and taking $M:=\gamma +1$, we obtain that
\begin{equation*}
 |G_{\theta, h}(x,t) \cap A_{M\kappa}\cap G^-_{\theta, h_0}(x_0,t_0)| 
  \geq  \underbrace{\frac{2^{-d-1} \lambda^d}{(1+\frac{1}{\gamma +1}+\Lambda d)^{d+1}} \frac{c}{\theta}}_{:=\sigma} |G_{\theta, h}(x,t)|,
\end{equation*}
as desired. \qedhere
\end{step}
\end{proof}

We next present the proof of the parabolic $W^{2,\eps}$ estimate.

\begin{proof} [Proof of Proposition~\ref{W_2eps_para}]
We begin with four reductions. 
First, we may assume that $L=0$. Otherwise we replace $u$ by $\hat u:=u+Lt$, which is solution of 
\begin{equation*}
 \partial_t\hat u+ \mathcal{P}^+_{\lambda,\Lambda} (D^2\hat u) \geq 0.
\end{equation*}
Let $(x,t)\in Q_1$ and $A>\udl \Theta(\hat u, Q_1)(x,t)$. 
By the definition of $\udl \Theta(\hat u, Q_1)(x,t)$, 
\begin{equation*}
  u(y,s)+Ls \geq  u(x,t) +Lt  +p\cdot(x-y)  - A\left( \frac{1}{2} |x-y|^2 +(t-s)\right) .
\end{equation*}
Since $L\geq 0$ and $t-s \geq 0$, we deduce that, for all $(y,s) \in Q_1$, $s\leq t$,
\begin{equation*}
  u(y,s) \geq  u(x,t)  +p\cdot(x-y)  - A\left( \frac{1}{2} |x-y|^2 +(t-s)\right) .
\end{equation*}
Then $A \geq \udl \Theta(u, Q_1)(x,t)$, and so 
$\udl \Theta(\hat u, Q_1)(x,t)\geq \udl \Theta(u, Q_1)(x,t)$.
Under the assumption the estimate holds true for $\hat u$, we get
 \begin{align*}
        |\{ (x,t)\in Q_{1/2}\left( 0,-\tfrac{1}{4}\right): \udl \Theta (u, & Q_1) (x,t)>\kappa \}|   \\
 & \leq |\{ (x,t)\in Q_{1/2}\left( 0,-\tfrac{1}{4}\right):\udl \Theta (\hat u,Q_1)(x,t)>\kappa \}|  \\ 
 & \leq C \left(\frac{\kappa}{\sup_{Q_1} |\hat u|} \right)^{-\eps}. 
 \end{align*}
Observing that $\sup_{Q_1} |\hat u| \leq \sup_{Q_1} |u| + L$, we get the desired estimate for $u$. 

Next, using the positive homogeneity of $\mathcal{P}^+_{\lambda,\Lambda}$ and $\udl \Theta$ and replacing~$u$ by $\tilde u := u/\sup_{Q_1}|u|$, we may assume that $\sup_{Q_1}|u|= 1$. Finally, by Lemma~\ref{comparison_Aop_Theta}, we have  
 \begin{equation*}
|\{ (x,s)\in Q_{1/2}\left( 0,-\tfrac{1}{4}\right):\udl \Theta (u,Q_1)(x,s)>\kappa \}| 
\leq  \left| Q_{1/2} \left(0,-\tfrac{1}{4} \right) \setminus A_\kappa \right| .
 \end{equation*}
  Thus, it suffices to prove that there exist some universal constants $C$, $\kappa_0$, $\eps>0$ such that for all~$\kappa \geq \kappa_0$, 
\begin{equation}\label{w2eps_Aop}
 \left| Q_{1/2} \left(0,-\tfrac{1}{4} \right) \setminus A_\kappa \right| \leq C  \left(\frac{\kappa}{\kappa_0} \right)^{-\eps} . 
\end{equation}

  \setcounter{step}{0}
\begin{step}[Geometric configuration] \label{step0}
Fix $\frac{3}{4} \leq \theta \leq 5$. We consider the cube 
\begin{equation*}
 C_R(0,T_0):=\left(0+(-R,R)^d\right) \times (T_0 -  H_\text{cube},T_0], 
\end{equation*}
with $T_0\in (-\tfrac{1}{2}, - \tfrac{1}{4}]$, $R>0$ which satisfies 
\begin{equation}\label{up_bound_R}
 R< \frac{1}{(2+16 \theta) \sqrt{d}}, 
\end{equation}
and $H_\text{cube}$, depending on $R$ and $\theta$, given by 
\begin{equation*}
H_\text{cube}:= (1+5\theta) d R^2.  
\end{equation*}
To perform our analysis, we introduce two parabolic balls $G_{1/2,H_1}(0,T_1)$ and $G_{1/2,H_2}(0,T_2)$ (see Figure~\ref{dessin_par_balls_cube})
where the parameters $T_1$, $H_1$, $T_2$ and $H_2$ denote 
\begin{align*}
 & H_1:= \tfrac{1}{2} dR^2 \quad \text{and}\quad  T_1:= T_0 + 4d\theta R^2,  \\
 & H_2:= \tfrac{1}{8} dR^2 \quad \text{and}\quad  T_2:=  T_1 +H_1 - H_2 = T_1 + \tfrac{3}{8} dR^2,  
\end{align*}
and a parabolic ball $G^-_{\theta,H_0} (x_0,t_0)$ with $(x_0,t_0)$ selected at Step~\ref{step1} and 
\begin{equation*}
 H_0  : = T_1-T_0 + H_1+  H_\text{cube} =\left(\tfrac{3}{2} +9\theta \right) dR^2. 
\end{equation*}
To be well-defined, our argument requires that $G_{1/2,H_1}(0,T_1)$ and $G_{1/2,H_2}(0,T_2)$ are contained in~$Q_1$. In particular, we need to check the condition $T_1+H_1 < 0$. 
By using the explicit expressions of $T_1$, $H_1$ and~\eqref{up_bound_R}, we get
\begin{equation*}
T_1+H_1= T_0 + 4 d \theta R^2  + \tfrac{1}{2} d R^2 \leq  - \tfrac{1}{4} +   \left( \tfrac{1}{2} + 4\theta \right) d R^2 < 0.
\end{equation*}
\end{step}

\begin{figure}[t]
\begin{center}
 \scalebox{0.93}{
\includegraphics{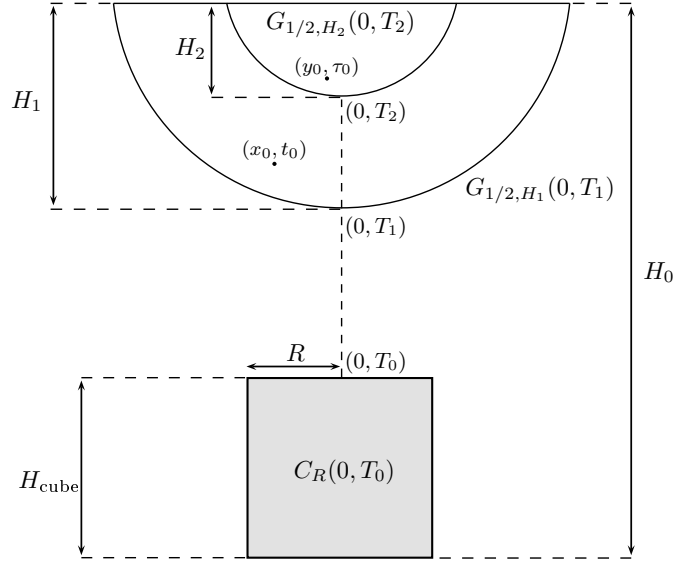} 
 }
\end{center}
\caption{Geometry in the proof of Proposition~\ref{W_2eps_para}: the cube $C_R(0,T_0)$ and the parabolic balls~$G_{1/2,H_1}(0,T_1)$ and $G_{1/2,H_2}(0,T_2)$ such that $C_R(0,T_0) \subseteq G_{\theta,H_0}^-(x_0,t_0)$ with $(x_0,t_0)\in G_{1/2,H_1}(0,T_1)$.
} 
\label{dessin_par_balls_cube} 
\end{figure}

\begin{step} [Existence of the paraboloid for a certain $\kappa=\kappa_0$]   \label{step1}
We claim that there exists $(x_0,t_0) \in G_{1/2,H_1}(0,T_1)$ such that
there exists $(y_0,s_0) \in B_1 \times (-1,t_0]$ such that
\begin{multline}\label{good_initial_point_parab}
 u(x_0,t_0) - \inf_{Q_1} u +\tfrac{\kappa}{2}|x_0-y_0|^2 - \kappa (t_0 - s_0) \\
 =\inf_{(z,\tau) \in Q_1, \tau \leq t_0} \left[ u(z,\tau) - \inf_{Q_1} u +\tfrac{\kappa}{2}|z-y_0|^2 - \kappa (\tau - s_0)  \right]=0.
\end{multline}
In addition,  
\begin{equation} \label{inclusion_cube_par_ball}
C_R(0,T_0) \subseteq G^-_{\theta,  H_0} (x_0,t_0) \subseteq Q_1, 
\end{equation}
and there exists $0<\xi<1$, depending only on $\theta$, such that
\begin{equation}\label{rapport_cylindre_boule_para}
  |C_R(0,T_0)| \geq \xi |G^-_{\theta,  H_0} (x_0,t_0)|.
\end{equation}
To prove the claim given by \eqref{good_initial_point_parab}, first we are going to find $(x_0,t_0)$ realizing the infimum 
for a good choice of $y_0$ and 
$\kappa$ chosen sufficiently large. Then we are going to determine $s_0$ such that this infimum will be equal to zero and we are going to check 
that $-1< s_0 \leq t_0$ to complete the proof.

First select $(y_0,\tau_0)\in G_{1/2,H_2}(0,T_2)$ such that  
\begin{equation*}
 u(y_0,\tau_0)=\inf \left\{ u(z,\tau) : (z,\tau) \in G_{1/2,H_2}(0,T_2) \right\}, 
\end{equation*}
and $\kappa$ such that
\begin{equation}\label{choice_op_init}
 \kappa \geq \kappa_0  \quad \text{with} \quad \kappa_0:= \max \left\{24,  \frac{320}{dR^2} \right\}. 
\end{equation}
Then for all $(z,\tau) \in Q_1$, 
\begin{equation*}
u(z,\tau)+\frac{\kappa}{2} |z-y_0|^2 -\kappa \tau \geq u(y_0,\tau_0)-\kappa \tau_0 -\osc_{Q_1} u +\kappa \left(\frac{1}{2} |z-y_0|^2-(\tau-\tau_0)\right).  
\end{equation*}
We check that $T_2 - T_1 >2 H_2$.
For all $(z,\tau)\in Q_1 \setminus G_{1/2,H_1}(0,T_1)$, $\tau \leq \tau_0$, we have  
 \begin{equation*}
   \tfrac{1}{2} |z-y_0|^2 -(\tau -\tau_0) \geq 
 \begin{cases}
\frac{1}{2} \left(\sqrt{T_2 - T_1} - \sqrt{2H_2}\right)^2 , 
& \text{if } \tau \in \left[\tfrac{1}{2} \left(T_1 + T_2\right),  \tau_0\right],  \\
  -\tfrac{1}{2} \left(T_1 + T_2\right) + \tau_0,  & \text{if } \tau < \tfrac{1}{2} \left(T_1 + T_2 \right). 
 \end{cases} 
 \end{equation*}
By inserting the values of $T_1$, $T_2$ and by using that $\tau_0\in G_{1/2,H_2}(0,T_2)$, we get
  \begin{equation*}
   \tfrac{1}{2} |z-y_0|^2 -(\tau -\tau_0) \geq 
 \begin{cases}
\frac{1}{160} dR^2 , 
& \text{if } \tau \in \left[\tfrac{1}{2} \left(T_1 + T_2\right),  \tau_0\right],  \\
 \tfrac{3}{16} dR^2 ,  & \text{if } \tau < \tfrac{1}{2} \left(T_1 + T_2\right). 
 \end{cases} 
 \end{equation*}
Since $\osc_{Q_1} u \leq 2$, and recalling \eqref{choice_op_init}, this implies that
\begin{multline*}
 \inf \left\{ u(z,\tau)+\tfrac{\kappa}{2} |z-y_0|^2- \kappa \tau :(z,\tau) \in Q_1 \setminus G_{1/2,H_1}(0,T_1), \tau \leq \tau_0 \right\}    \\
 \geq u(y_0,\tau_0)-\kappa \tau_0.
\end{multline*}
Thus there exists $(x_0,t_0) \in G_{1/2,H_1}(0,T_1) $, $t_0\leq \tau_0$, such that
\begin{align*}
u(x_0,t_0) +\tfrac{\kappa}{2} |x_0-y_0|^2 - \kappa t_0
& = \inf_{(z,\tau) \in Q_1, \tau \leq \tau_0} \left( u(z,\tau) + \tfrac{\kappa}{2} |z-y_0|^2-\kappa \tau  \right)\\
& = \inf_{(z,\tau) \in Q_1, \tau \leq t_0} \left( u(z,\tau) + \tfrac{\kappa}{2} |z-y_0|^2-\kappa \tau  \right).
\end{align*}
Now let $s_0 \leq t_0$ be defined by 
\begin{equation*}
s_0: =t_0 - \frac{1}{\kappa} \left( u(x_0,t_0) -\inf_{Q_1} u\right) -\frac{1}{2} |x_1-y_0|^2.
\end{equation*}
Thus, the infimum in \eqref{good_initial_point_parab} is equal to zero.
To complete the proof of the claim, it remains to check that $s_0 > -1$. By using that $u(x_0,t_0)-\inf_{Q_1} u \leq 2$, we get
\begin{equation*}
t_0 - s_0 \leq   \frac{2}{\kappa} + \frac{1}{2} |x_0-y_0|^2.
\end{equation*}
Since $(x_0,t_0),(y_0,\tau_0) \in G_{\theta, H_1} (0,T_1)$, we get by inserting \eqref{up_bound_R} and using $\theta \geq 3/4$  that
\begin{equation*}
|x_0-y_0|^2 \leq \frac{2 H_1}{\theta}= \frac{dR^2}{\theta} 
\leq  \frac{1}{4\theta(1+8 \theta)^2} \leq \frac{1}{147} \leq \frac{1}{6}. 
\end{equation*}
Thus we obtain
\begin{equation*}
t_0 - s_0 \leq \frac{2}{\kappa} + \frac{1}{12}. 
\end{equation*}
By using~\eqref{choice_op_init}, we get 
\begin{equation*}
 t_0 - \tfrac{1}{6} \leq s_0 \leq t_0. 
\end{equation*}
Since $t_0 \in [T_1,T_1+H_1]=[T_0+4d\theta R^2, T_0 +\left(4\theta +\tfrac{1}{2}) dR^2\right]$, in particular, $-\tfrac{1}{2}\leq T_0\leq t_0\leq 0$ 
and we conclude that $s_0> - 1$. This completes the proof of \eqref{good_initial_point_parab}. 

To show \eqref{inclusion_cube_par_ball}, first notice that $|x_0|\leq \sqrt{d}R$ since $(x_0,t_0)\in G_{1/2,H_1}(0,T_1)$. 
Now observe that each $(y,s) \in C_R(0,T_0)$ satisfies $t_0 -s\leq H_0$ and  
\begin{equation*}
|y-x_0| \leq |y| + |x_0|\leq \sqrt{d} R+\sqrt{d} R = 2\sqrt{d} R.  
\end{equation*}
This means $(y,s)\in G^-_{\theta,H_0}(x_0,t_0)$ since $T_1 - T_0\geq 4d\theta R^2$ (see also Figure~\ref{dessin_par_balls_cube}). 
By a direct computation, we check~\eqref{rapport_cylindre_boule_para}.
\end{step}  
  
 \begin{step}\label{step2}   
By Step~\ref{step1}, the point $(x_0,t_0)$ belongs to $A_{\kappa_0}$. By Lemma \ref{comparison_Aop_Theta}, we deduce that, for all $\kappa \geq \kappa_0$, 
$(x_0,t_0)$ belongs to $A_{\kappa}$. Then we can apply Lemma~\ref{measure_estimate_macroscopic_ball} and we get that for all $\kappa \geq \kappa_0$, 
 \begin{equation*}
 |A_{M\kappa} \cap G^{-}_{\theta,H_0}(x_0,t_0)|\geq |G^{-}_{\theta,H_0}(x_0,t_0) \cap A_{\kappa}|
                                            +\sigma \eta |G^{-}_{\theta,H_0}(x_0,t_0) \setminus A_{\kappa} |.
 \end{equation*}
After rearranging the terms, this implies
  \begin{equation} \label{measure_control_one_iteration}
  | G^{-}_{\theta,H_0}(x_0,t_0) \setminus A_{M\kappa}| \leq  (1-\sigma \eta)| G^{-}_{\theta,H_0}(x_0,t_0) \setminus A_{\kappa}|.
  \end{equation}
\end{step}

\begin{step}\label{step3}
   We claim that for all $\kappa>\kappa_0$, we have
 \begin{equation*} 
| C_R(0,T_0) \setminus A_\kappa| \leq \frac{1}{\xi (1- \sigma \eta)} \left(\frac{\kappa}{\kappa_0 }\right)^{-\eps} |C_R(0,T_0)|,
 \end{equation*} 
where $\eps$ is given by $\eps := - \frac{\ln (1- \sigma \eta)}{\ln M}>0$.
First we obtain the decay measure estimate on the parabolic ball $G^{-}_{\theta,H_0}(x_0,t_0)$.  
Let $\kappa>\kappa_0$ and $N$ be the integer defined by
 \begin{equation*}
M^N\kappa_0<\kappa \leq M^{N+1}\kappa_0 
\quad \Longleftrightarrow \quad N:=\left\lceil \tfrac{1}{\ln M}\ln \left(\tfrac{\kappa}{\kappa_0} \right) \right\rceil -1. 
 \end{equation*}
 Here $\left\lceil r  \right\rceil$ denotes, for $r\in \R$, the smallest integer not smaller than $r$.
 Then by using iteratively~\eqref{measure_control_one_iteration} given by Step~\ref{step2}, we deduce that 
 \begin{equation*}
| G^{-}_{\theta,H_0}(x_0,t_0)\setminus A_\kappa| \leq (1- \sigma \eta)^N | G^{-}_{\theta,H_0}(x_0,t_0) \setminus A_{\kappa_0}|.
\end{equation*}
By inserting the value of $N$, we conclude that
\begin{equation*}
|G^{-}_{\theta,H_0}(x_0,t_0) \setminus A_\kappa|
\leq \frac{1}{1-\sigma \eta} |G^{-}_{\theta,H_0}(x_0,t_0)| \left(\frac{\kappa}{\kappa_0}\right)^{-\eps} ,
 \end{equation*}
By combining \eqref{inclusion_cube_par_ball} and \eqref{rapport_cylindre_boule_para} (see Step~\ref{step1}), 
 we come back to the cube $C_R(0,T_0)$ and obtain the desired estimate. 
\end{step}

\begin{step}[Covering argument] \label{step4}
We cover $Q_{1/2}$ by 
\begin{equation*}
 Q_{1/2}\left(0,-\tfrac{1}{4} \right)  \subseteq \bigcup_{1\leq i \leq N} \ol{C}_{R} (X_i,T_i) \subseteq Q_1
 \end{equation*}
where the parabolic cylinders $C_{R} (X_i,T_i):=C_{R} (0,T_0) + (X_i,T_i) - (0,T_0)$ are disjoint. Then 
  \begin{equation*}
\left| Q_{1/2}\left(0,-\tfrac{1}{4} \right)  \setminus A_\kappa \right| \leq  \sum_{i=1}^N | C_R(X_i,T_i) \setminus A_\kappa |.
  \end{equation*}
  By applying Step~\ref{step3} to each parabolic cylinder $C_R(X_i,T_i)$, we get
  \begin{equation*}
   \left| Q_{1/2}\left(0,-\tfrac{1}{4} \right)  \setminus A_\kappa \right|
  \leq \frac{1}{\xi (1-\sigma \eta)} \sum_{i=1}^N |C_R(X_i,T_i)| \left(\frac{\kappa}{\kappa_0}\right)^{-\eps} .
  \end{equation*}
  Since $\ds \sum_{i=1}^N |C_R(X_i,T_i)| \leq |Q_1|$, we get the estimate given by \eqref{w2eps_Aop}.  \qedhere
\end{step}  
\end{proof}

\section{Proof of Theorem~\ref{W3_eps_para}}
\label{part_par_reg_proof_s}

To prove Theorem~\ref{W3_eps_para}, we differentiate the equation to obtain the result from the parabolic $W^{2,\eps}$~estimate obtained in Section~\ref{w_2_eps_par}. 
In the elliptic case, the proof of the elliptic  $W^{3,\eps}$~estimate  strongly uses the $C^{1,\alpha}$~estimates in order to apply $W^{2,\eps}$~estimates on the components of the gradient~$Du$.  
In the parabolic case, the $C^{1,\alpha}$~estimates do not imply that $u$ is differentiable with respect to the time variable. Thus the main new challenge which arises is to upgrade the regularity with respect to time, which is accomplished in Proposition~\ref{key_prop_para_ref}. 

The idea of the proof of Proposition~\ref{key_prop_para_ref} is to separate time and space in order to gain local regularity. First we are going to obtain regularity in space for a fixed time $s$, by applying the~$W^{2,\eps}$~estimate on the derivatives of~$u$. 
This step gives good quadratic approximations in space and, since $u$ solves the PDE, we obtain a first-order approximation of the solution with respect to time. Instead of obtaining directly the estimate like in the elliptic case (see~\cite{armstrong_smart_silvestre}), we proceed by contradiction by considering a local maximum (or minimum) and we use viscosity solution arguments and the uniform ellipticity of the operator. 

Note that we use $u$ is solution of the PDE to obtain a connection between $\Psi(u)$ and $\Theta(u_{x_i})$, $1 \leq i \leq d$, which is different than the elliptic case.

\begin{prop} \label{key_prop_para_ref}
Assume that $F$ satisfies~\ref{f1}, $F(0)=0$ and $g(0,0)=0$.  
Suppose that $u\in C(Q_1)$ is a viscosity solution of~\eqref{eq_para_unif_ell} in $Q_1$ such that $\sup_{Q_1} |u|\leq 1$. 
There exists a universal constant~$C_1>0$ such that, for every $\kappa= (\kappa_1,\cdots, \kappa_d) \in \R_+^d$ such that $|\kappa| \geq 1$,
\begin{multline*}
\bigcap_{i=1}^d \left\{(y,s) \in Q_{1/2}\left(0,-\tfrac{1}{4}\right) : \Theta (u_{x_i}, Q_1)(y,s) \leq \kappa_i \right\}   \\
 \subseteq \left\{(y,s) \in Q_{1/2}\left(0,-\tfrac{1}{4}\right) : \Psi(u,Q_1)(y,s)
  \leq C_1 \left( 1+\left\|g \right\|_{C^{0,1}(Q_1)} \right)|\kappa| \right\}.
\end{multline*}
\end{prop}

\begin{proof} Fix $\kappa= (\kappa_1,\cdots, \kappa_d) \in \R_+^d$ and  $(x_0,t_0) \in Q_{1/2}\left(0,-\tfrac{1}{4}\right)$ 
such that, for every $1\leq i \leq d$, 
\begin{equation*}
 \Theta (u_{x_i}, Q_1) (x_0,t_0) \leq \kappa_i.
\end{equation*}
For all $1\leq i \leq d$, we select $p^i\in \R^d$ such that, for every $y\in B_1$ and $-1 < s\leq  t_0$, 
\begin{equation}\label{gradient_p_i_estimate}
 |u_{x_i}(y,s) - u_{x_i}(x_0,t_0) - p^i\cdot(y - x_0)| \leq \kappa_i \left(\tfrac{1}{2}   |x_0-y|^2 + t_0 - s \right).
\end{equation}
Define the polynomial approximation $P$ to $u$ given by
\begin{equation}\label{def_poly_approx}
P(y,s)= u(x_0,t_0) + b(s-t_0)+ Du(x_0,t_0)\cdot (y-x_0) +\tfrac{1}{2} \langle (y-x_0), M (y-x_0)\rangle
\end{equation}
where $b$ and $M$ are going to be chosen. First we set $M:=(p^i_j)\in \mathbb{M}_d$. 
 Then to fix $b$, up to replacing~$M$ by $\frac{M+M^\top}{2}$ in \eqref{def_poly_approx}, we can assume that $M\in \mathbb{S}_d$ and then we take 
\begin{equation}\label{choix_B_M}
 b:= - F\left(M\right) +g(x_0,t_0).
\end{equation}
To estimate the difference between $P$ and $u$, we separate the difference into two parts, a space term and a time term 
\begin{multline*}
|u(y,s)-P(y,s)|  \leq | u(y,s) - u(x_0,s) - Du(x_0,t_0)\cdot (y-x_0) -\tfrac{1}{2} \langle (y-x_0), M (y-x_0)\rangle  | \\
  +  |u(x_0,s) - u(x_0,t_0) - b(s-t_0)|.
\end{multline*}
We proceed in eight steps. Steps \ref{step3eps1}--\ref{step3eps2} provide the upper bound on the space term. Then by introducing an adequate test function and arguing by a comparison principle argument, 
Steps \ref{step3eps3}--\ref{step3eps7} will give the corresponding upper bound on the time term.
Finally, we will conclude in Step 8. 

 \setcounter{step}{0}
 \begin{step} \label{step3eps1}
 We claim that for all $1\leq i \leq d$, $(y,s)\in Q_1$ with  $s\leq t_0$, 
\begin{equation}\label{theta_est_gradient}
   |u_{x_i}(y,s)  - u_{x_i}(y,t_0)| \leq  \kappa_i \left( |x_0-y|^2 + t_0-s    \right).
 \end{equation}
To prove this, notice that 
\begin{equation*} 
 |u_{x_i}(y,t_0) - u_{x_i}(x_0,t_0) - p^i\cdot(y-x)| \leq  \tfrac{\kappa_i}{2}   |x_0-y|^2.
\end{equation*}
By combining this inequality with \eqref{gradient_p_i_estimate}, we get 
\begin{align*} 
  |u _{x_i}& (y,s)  - u_{x_i}(y,t_0)| \\
& \leq |u_{x_i}(y,s) - u_{x_i}(x_0,t_0) - p^i\cdot(y-x_0)| + | u_{x_i}(y,t_0) - u_{x_i}(x_0,t_0) - p^i\cdot(y-x_0)| \\
& \leq   \kappa_i \left(\tfrac{1}{2}   |x_0-y|^2 + t_0-s  \right) +  \tfrac{\kappa_i}{2}   |x_0-y|^2 .
\end{align*}
The proposed estimate \eqref{theta_est_gradient} directly follows.  
 \end{step}
\begin{step} \label{step3eps2} 
  We next prove the following ``slice estimate'': for every $(y,s)\in Q_1$ such that $-1 < s\leq t_0$, we have
 \begin{multline}\label{ineq_space_after_young}
 \left| u(y,s) - u(x_0,s) - Du(x_0,t_0) \cdot (y-x_0) - \tfrac{1}{2} \langle (y-x_0), M(y-x_0) \rangle  \right|  \\  
  \leq |\kappa| \left(\tfrac{7}{6} |y-x_0|^3+  \tfrac{\sqrt{2}}{3} |s-t_0|^{3/2} \right).
\end{multline}
Since $u\in C^1$ with respect to the space variable, we can write 
\begin{align*}
I:&=\left| u(y,s) - u(x_0,s) - Du(x_0,t_0) \cdot (y-x_0) - \tfrac{1}{2} \langle (y-x_0), M(y-x_0) \rangle  \right|  \\
  &=\left| (y-x_0) \cdot \int_0^1 Du(x_0+\tau(y-x_0), s) - Du(x_0,t_0) - \tau M(x_0-y_0) d\tau  \right|. 
\end{align*}
It is clear that $I \leq I_1 + I_2$ where $I_1$ and $I_2$ respectively denote 
\begin{equation*}
 I_1:= \left|(y-x_0) \cdot \int_0^1 Du(x_0+\tau(y-x_0), s)  - Du(x_0+\tau(y-x_0)), t_0) d\tau  \right| 
\end{equation*}
and 
\begin{equation*}
I_2:= \left| (y-x_0)  \cdot \int_0^1 Du(x_0+\tau(y-x_0), t_0) - Du(x_0,t_0) - \tau M(x_0-y_0) d\tau  \right|.
\end{equation*}
It remains to determine some upper bounds on $I_1$ and $I_2$.
For $I_1$, applying successively Cauchy-Schwarz inequality and Step 1 yields that 
\begin{align*}
I_1 & \leq |y-x_0| \int_0^1 \left| Du(x_0+\tau(y-x_0),s) - Du(x_0+\tau(y-x_0)), t_0)\right|  d\tau   \\
      & \leq |y-x_0| \int_0^1 |\kappa| (\tau^2|y-x_0|^2+ t_0-s) d\tau \\ 
      & \leq  |\kappa| \left(\tfrac{1}{3}|y-x_0|^3+|y-x_0|  |s-t_0| \right).
\end{align*}
We estimate $I_2$ by using~\eqref{gradient_p_i_estimate} and the same computations than those used in~\cite{armstrong_smart_silvestre}.
For sake of completeness and reader convenience, we give here the arguments. 
According to~\eqref{gradient_p_i_estimate}, 
\begin{equation*}
 |u_{x_i}(x_0+\tau(y-x_0),t_0) - u_{x_i}(x_0,t_0) - \tau p^i\cdot(y-x_0)| \leq  \frac{1}{2} \kappa_i \tau^2  |y-x_0|^2 .
\end{equation*}
By Cauchy-Schwarz inequality,
\begin{equation*}
I_2\leq  |y-x_0| \int_0^1  \frac{1}{2} \left(\sum_{i=1}^d \kappa_i^2\right)^{1/2} \tau^2  |y-x_0|^2  d\tau 
    = \tfrac{|\kappa|}{6} |y-x_0|^3 .
\end{equation*}
To obtain~\eqref{ineq_space_after_young}, it suffices to apply Young inequality which yields
\begin{equation*}
|y-x_0|   |s-t_0| \leq \frac{2}{3} |y-x_0|^3 +\frac{\sqrt{2}}{3}  |s-t_0|^{3/2}, 
\end{equation*}
and a simple calculation gives~\eqref{ineq_space_after_young}.
\end{step}
To complete the proof of Proposition \ref{key_prop_para_ref}, it remains now to get a similar estimate for the term in time.
We will show that there exists a universal constant $C_2 \geq 1$ such that, for all $-1 < s' \leq t_0$, 
\begin{equation} \label{ineq_time_x0}
 \left|  u(x_0,s') - u(x_0,t_0) - b(s'-t_0) \right| \leq  C_2  \left(1+\left\|g \right\|_{C^{0,1}(Q_1)}\right)  |\kappa| |t_0 - s'|^{3/2}.
\end{equation}
We are going to prove that, for all $-1 < s' \leq t_0$, 
\begin{equation} \label{ineq_time_x0_lower_bound}
  u(x_0,s') - u(x_0,t_0) - b(s'-t_0) \geq  - C_2 \left(1+\left\|g \right\|_{C^{0,1}(Q_1)} \right) |\kappa| |t_0 - s'|^{3/2}, 
\end{equation}
the argument for the reverse inequality is entirely parallel. 
To do this, we will consider a suitable test function taking into account the size of $b$.
\begin{step}\label{step3eps3} 
We claim that the following a priori bound on $b=-F(M)+g(x_0,t_0)$ holds 
\begin{equation}\label{borne_B_gamma}
 |b| \leq  c_2 \left(|\kappa|+1 +\left\|g \right\|_{C^{0,1}(Q_1)} \right) ,  
\end{equation}
where $c_2>0$ a universal constant. \\
By considering \eqref{gradient_p_i_estimate} for $s=t_0$ and selecting $y\in \ol B_{3/4} \subseteq B_1$ 
such that $ x_0-y = \frac{1}{4} \frac{p^i}{|p^i|}$ if $p^i \neq 0$, we get
\begin{multline*}\label{}
\tfrac{1}{4} |p^i| - \sup_{y\in \ol B_{3/4}} |u_{x_i}(y,t_0) - u_{x_i}(x_0,t_0)| \\
\leq  |u_{x_i}(y,t_0) - u_{x_i}(x_0,t_0) - p^i\cdot(x_0-y)| \leq \tfrac{\kappa_i }{2}   |x_0-y|^2 .
\end{multline*}
By rearranging the terms, we obtain
\begin{equation*}
\tfrac{1}{4} |p^i| \leq  \tfrac{9}{32}\kappa_i + \sup_{y\in \ol B_{3/4} } |u_{x_i}(y,t_0) - u_{x_i}(x_0,t_0)|
\leq  \tfrac{9}{32}\kappa_i +  2 \sup_{\ol Q_{3/4}} |Du|.
\end{equation*}
Now by Proposition~\ref{int_reg_parabolic}, there exists a universal constant $\widetilde C>0$ such that
\begin{equation*}\label{int_para_estimate_o1}
\sup_{\ol Q_{3/4}} |Du|  \leq \widetilde C \left(1+\left\|g \right\|_{C^{0,1}(Q_1)} \right). 
\end{equation*}
We deduce that there exists a universal constant $c>9/8$ such that, for all $1\leq i \leq d$, 
\begin{equation*}
 |p^i| \leq c \left(\kappa_i+ 1 + \left\|g \right\|_{C^{0,1}(Q_1)} \right). 
\end{equation*}
If $\norm{M}$ denotes the $(L^2,L^2)$-norm of $M$, i.e. $\norm{M}= \sup_{|x|=1}|Mx|$, then our choice for $M=(p^i_j)$ implies that 
\begin{equation*}
 \norm{M} \leq d c \left(|\kappa| +1 +\left\|g \right\|_{C^{0,1}(Q_1)} \right). 
\end{equation*}
Since $F$ satisfies~\ref{f1} and $F(0)=0$, it is immediate to check that for all $N\in \mathbb{S}_d$, $|F(N)|\leq d\Lambda \norm{N}$. 
Thus, by using \eqref{choix_B_M}, $g(0,0)=0$ and  $(x_0,t_0) \in Q_1$, we get 
\begin{equation*} 
 |b| \leq d \Lambda \norm{M} +|g(x_0,t_0)| \leq d \Lambda \norm{M} +\left\|g \right\|_{C^{0,1}(Q_1)},
\end{equation*}
and the claim \eqref{borne_B_gamma} easily follows.
\end{step}

\begin{step}\label{step3eps4}  
Let $s' \in (-1,t_0]$. Next we show that  
  \begin{equation*} 
   \phi \geq u \qquad \text{holds on }\partial_{p} \left(B_{3/4} \times \{s'< s \leq  t_0 \} \right), 
 \end{equation*}
where we have defined
 \begin{multline*}
 \phi(y,s) :=  u(x_0,s')    +  Du(x_0,t_0) \cdot (y-x_0) +\tfrac{1}{2}  \langle (y-x_0), M(y-x_0) \rangle \\
 + \left(b+\left( \tfrac{C_0\beta}{4} + 2 \left\|g \right\|_{C^{0,1}(Q_1)} \right) |\kappa| |s'-t_0|^{1/2} \right) (s-s') 
  + \tfrac{C_0}{6} |\kappa| (|y-x_0|^3 + |s'-t_0|^{3/2})  
 \end{multline*}
 where $\beta:=\max \{1, 2\Lambda (d+1) \}$ and $C_0$ is a constant
depending on $d$, $\lambda$, $\Lambda$ and $\left\|g \right\|_{C^{0,1}(Q_1)}$ given by 
\begin{equation}\label{cond_C0}
C_0 := 768 \left(c_2 \left(1+ \left\|g \right\|_{C^{0,1}(Q_1)}\right)  +3 \right).
\end{equation}
First we claim that $\phi \geq u $ on $\ol B_{3/4} \times \{s'\}$.
To prove this, it follows from~\eqref{ineq_space_after_young} that for all $y\in \ol B_{3/4} \subseteq B_1$, 
  \begin{multline*} 
 u(y,s') \leq  u(x_0,s') + Du(x_0,t_0) \cdot (y-x_0) + \tfrac{1}{2} \langle (y-x_0), M(y-x_0) \rangle    \\
 + |\kappa|  \left(\tfrac{7}{6} |y-x_0|^3+  \tfrac{\sqrt{2}}{3} |s'-t_0|^{3/2} \right).
 \end{multline*}
Since $\tfrac{C_0}{6} \geq  \max \left\{\tfrac{7}{6}, \tfrac{\sqrt{2}}{3} \right\}$, this yields that 
\begin{equation*}
 u(\cdot, s')\leq \phi(\cdot, s')  \quad \text{on } \ol B_{3/4}. 
\end{equation*}
Then we claim that $\phi \geq u $ on $\{(y,s) : |y|=3/4 , s'< s \leq t_0  \}$.
Arguing by contradiction, assume that there exists $(x_1,t_1)$ with $|x_1|=3/4$, $s'<t_1\leq t_0$ such that
\begin{equation*}
\phi(x_1,t_1) < u(x_1,t_1).
\end{equation*}
Then, 
\begin{equation*}
\phi(x_1,t_1) - u(x_1,s') <  |u(x_1,t_1) - u(x_1,s')| \leq  2.
\end{equation*}
Moreover,
\begin{multline*}
\phi(x_1,t_1) - u(x_1,s') 
= \tfrac{C_0}{6} |\kappa| \left(|x_1-x_0|^3 + |s'-t_0|^{3/2}  \right)  \\
   + u(x_0,s') - u(x_1,s')  +Du(x_0,t_0) \cdot (x_1-x_0) 
 +\tfrac{1}{2}  \langle (x_1-x_0), M(x_1-x_0) \rangle \\
 + \left(b+ \left( 2 \left\|g \right\|_{C^{0,1}(Q_1)}  +\tfrac{C_0}{4}\beta \right) |\kappa| |s'-t_0|^{1/2}\right) \underbrace{(t_1-s')}_{\geq 0} .
\end{multline*}
Neglecting the nonnegative time terms, we obtain  
\begin{multline*}
\phi(x_1,t_1) - u(x_1,s') \geq \tfrac{C_0}{6} |\kappa| |x_1-x_0|^3 - |b| |t_1-s'| \\
    - |u(x_0,s') - u(x_1,s') +Du(x_0,t_0) \cdot (x_1-x_0) +\tfrac{1}{2}  \langle (x_1-x_0), M(x_1-x_0) \rangle| .
\end{multline*}
By combining the slice estimate \eqref{ineq_space_after_young} and the bound on $b$ given by \eqref{borne_B_gamma}, 
we obtain that 
\begin{multline*}
\phi(x_1,t_1) - u(x_1,s') \geq   \tfrac{C_0}{6}  |\kappa| |x_1-x_0|^3 \\ 
-c_2 \left(|\kappa|+1 +  \left\|g \right\|_{C^{0,1}(Q_1)} \right) |t_1-s'|  
                                -|\kappa| \left(\tfrac{7}{6} |x_1-x_0|^3+  \tfrac{\sqrt{2}}{3} |t_1-t_0|^{3/2} \right). 
\end{multline*}
Using that $\tfrac{1}{4}\leq |x_1-x_0| \leq \tfrac{5}{4}$, $|t_1-t_0|\leq 1$ and $|t_1-s'|  \leq 1$, the inequality above simply reduces to 
\begin{equation*}
\tfrac{C_0}{384} |\kappa| < \left(c_2 \left(1+ \left\|g \right\|_{C^{0,1}(Q_1)} \right) +3 \right)(|\kappa|+1). 
\end{equation*}
By inserting \eqref{cond_C0}, we obtain $2 |\kappa|  <|\kappa| + 1$, and we get a contradiction if $|\kappa| \geq 1$. 
\end{step}

\begin{step}\label{step3eps5} 
We claim that:
\begin{equation*} 
u- \phi \text{ can attain a positive global maximum only in }  \left\{(y,s) : |y-x_0| < |s-t_0|^{1/2} \right\}.  
\end{equation*}
By Step \ref{step3eps4}, we may assume that $u- \phi$ attains a positive global maximum at $(x_1,t_1) \in B_{3/4}\times \{s' < s \leq t_0\}$:
\begin{equation*} 
(u-\phi)(x_1,t_1) = \sup_{ B_{3/4} \times \{s' < s \leq  t_0 \}} (u-\phi)(y,s). 
 \end{equation*}
In particular,  
\begin{equation*}
  (u-\phi)(x_1,t_1) \geq (u-\phi)(x_0,t_1).
\end{equation*}
A direct computation yields
\begin{equation*}
 \phi(x_1,t_1) -\phi(x_0,t_1)= Du(x_0,t_0)\cdot (x_1-x_0)+\tfrac{1}{2} \langle (x_1-x_0), M(x_1-x_0) \rangle 
 + \tfrac{C_0}{6} |\kappa| |x_1-x_0|^3.
\end{equation*}
After  rearranging the terms, we obtain
\begin{multline}\label{space_minor_y0s0}
u(x_1,t_1)-u(x_0,t_1)- Du(x_0,t_0) \cdot (x_1-x_0)-\tfrac{1}{2} \langle (x_1-x_0),M(x_1-x_0) \rangle  \\
\geq \tfrac{C_0}{6} |\kappa| |x_1-x_0|^3.  
\end{multline}
Moreover, we know by \eqref{ineq_space_after_young} that 
\begin{multline*} 
 | u(x_1,t_1)  - u(x_0,t_1)  - Du(x_0,t_0) \cdot (x_1-x_0) - \tfrac{1}{2}  \langle (x_1-x_0), M(x_1-x_0) \rangle  |  \\ 
 \leq |\kappa|  \left( \tfrac{7}{6}  |x_1-x_0|^3   +\tfrac{\sqrt{2}}{3}  |t_1-t_0|^{3/2}  \right).
\end{multline*}
By using \eqref{space_minor_y0s0} we deduce from the inequality above that 
\begin{equation*}
 |\kappa|  \left(  \tfrac{7}{6}  |x_1-x_0|^3 +\tfrac{\sqrt{2}}{3} |t_1-t_0|^{3/2} \right) \geq \tfrac{C_0}{6} |\kappa| |x_1-x_0|^3 .  
 \end{equation*}
Since $\frac{C_0}{6} \geq \tfrac{7}{3}$, we obtain 
\begin{equation*}
 |x_1-x_0| \leq  \left( \tfrac{2 \sqrt{2} }{7}\right)^{1/3}|t_1-t_0|^{1/2}.  
\end{equation*}
In particular, this yields the desired claim.  
\end{step}

\begin{step}\label{step3eps6} 
We next show that $u- \phi$ cannot achieve any local maximum in the cylinder~$\widetilde Q$ given by 
\begin{equation*}
\widetilde Q : =\{(y,s) \in B_1 \times [s',t_0]: |y-x_0| <  |s'-t_0|^{1/2} \}, 
\end{equation*}
by arguing that $\phi$ is a strict supersolution in $\widetilde Q$ i.e., 
\begin{equation}\label{phi_strict_supersol}
 \partial_t \phi+ F(D^2\phi) > g, \quad \text{in } \widetilde Q.
\end{equation}
We verify \eqref{phi_strict_supersol} by a direct computation.
By the uniform ellipticity condition \ref{f1}, and noticing that the perturbation is a positive matrix, 
\begin{equation*}
  F\left(M + \frac{C_0 |\kappa| }{2} |y-x_0| \left[I_d+ \frac{y-x_0}{|y-x_0|} \otimes  \frac{y-x_0} {|y-x_0| } \right]  \right)
  \geq F(M) - \frac{\Lambda(d+1)}{2} C_0 |\kappa| |y-x_0|.
\end{equation*}
This inequality yields 
\begin{multline*}  
 b   + \left( 2 \left\|g \right\|_{C^{0,1}(Q_1)}  + \frac{\beta}{4} C_0 \right) |\kappa| |s'-t_0|^{1/2} \\
 + F\left(M+ \frac{C_0|\kappa|}{2} |y-x_0| \left[ I_d + \frac{y-x_0}{|y-x_0|} \otimes  \frac{y-x_0} {|y-x_0|}  \right] \right)  \\
  \geq b +2 \left\|g \right\|_{C^{0,1}(Q_1)} |\kappa| |s'-t_0|^{1/2} + F(M) + \frac{\beta}{4} C_0 |\kappa| |s'-t_0|^{1/2} - \frac{\Lambda(d+1)}{2}C_0 |\kappa| |y-x_0|.
\end{multline*}
For our choice given by \eqref{choix_B_M}, we have $b +F(M)=g(x_0,t_0)$. 
Recalling the value of $\beta$ and $|\kappa|\geq 1$, we obtain that the following inequality holds in $\widetilde Q$: 
\begin{multline*}
 b + \left(2 \left\|g \right\|_{C^{0,1}(Q_1)}  + \frac{\beta}{4} C_0 \right) |\kappa| |s'-t_0|^{1/2} \\
  + F\left(M+ \frac{C_0|\kappa|}{2} |y-x_0| \left[I_d + \frac{y-x_0}{|y-x_0|} \otimes  \frac{y-x_0} {|y-x_0|}  \right] \right)   \\
 > \left\|g \right\|_{C^{0,1}(Q_1)}\left(|y-x_0| + |s-t_0|^{1/2} \right) +g(x_0,t_0) \geq g(y,s).
\end{multline*}
This confirms \eqref{phi_strict_supersol}. 
\end{step}

\begin{step}\label{step3eps7} 
By Steps \ref{step3eps4}--\ref{step3eps6}, $u-\phi$ cannot achieve any positive global maximum. Hence, 
\begin{equation*} 
 u \leq \phi \quad \text{in } \ol B_{3/4}  \times [s', t_0] .
\end{equation*}
In particular, for $s=t_0$,  $u(x_0,t_0) \leq \phi(x_0,t_0)$, this yields
\begin{equation*}
 u(x_0,t_0) \leq  u(x_0,s') +b (t_0-s') 
 + \left(\tfrac{C_0 \beta}{4}+ 2 \left\|g \right\|_{C^{0,1}(Q_1)} \right) |\kappa| |t_0 - s'|^{3/2}+\tfrac{C_0}{6} |\kappa| |t_0 - s'|^{3/2}.
\end{equation*}
After rearranging the terms, we get
\begin{equation*}
 - \left(C_0 \left(\tfrac{1+\beta}{4} \right) + 2 \left\|g \right\|_{C^{0,1}(Q_1)} \right) 
 |\kappa| |t_0 - s'|^{3/2}  \leq  u(x_0,s') - u(x_0,t_0) - b(s'-t_0) .
\end{equation*}
Thus, by using \eqref{cond_C0}, it follows \eqref{ineq_time_x0_lower_bound} by taking $C_2:=192 (1+\beta) (c_2+3)$.
\end{step} 

\begin{step}\label{step3eps8} We conclude the argument. 
By combining both \eqref{ineq_space_after_young} and \eqref{ineq_time_x0}, we deduce that, for all $(y,s)\in Q_1$, $s\leq t_0$, 
\begin{align*}
|u(y,s)- & P(y,s)| \\
&  \leq |\kappa| \left(\tfrac{7}{6} |y-x_0|^3+  \tfrac{\sqrt{2}}{3} |s-t_0|^{3/2} \right) 
 +  C_2 \left(1+\left\|g \right\|_{C^{0,1}(Q_1)} \right) |\kappa| |t_0 - s|^{3/2} \\
 &   \leq \left(C_2 \left(1+\left\|g \right\|_{C^{0,1}(Q_1)} \right) +\tfrac{7}{6}  \right) |\kappa| \left(|y-x_0|^3 + |s-t_0|^{3/2} \right). %
\end{align*}
By setting $C_1:=C_2+\tfrac{7}{6}$, this implies 
$\Psi(u,Q_1)(x_0,t_0)\leq C_1 \left(1+\left\|g \right\|_{C^{0,1}(Q_1)} \right) |\kappa|$. \qedhere
\end{step}
\end{proof}
Now we can give the proof of the parabolic $W^{3,\eps}$ estimate stated in the introduction. 
\begin{proof}[Proof of Theorem~\ref{W3_eps_para}] 
If $u \equiv 0$ on $Q_1$, then the estimate is clear, so we may assume that $\sup_{Q_1}|u|>0$.
In Step~\ref{step1_th2}, we reduce the proof of the theorem to the case of $\sup_{Q_1}|u|\leq 1$, $g(0,0)=0$ 
and $F(0)=0$ by scaling arguments. 
In Step~\ref{step2_th2}, we prove the theorem under these assumptions. 
 \setcounter{step}{0}
\begin{step}\label{step1_th2}
We first reduce to the case that $g(0,0)=0$. If $g(0,0) \neq 0$, 
define $ \ol g(x,t):=g(x,t) - g(0,0)$ and $\ol u(x,t):= u(x,t) - t g(0,0)$ is a solution of  
\begin{equation*}
 \partial_t \ol u + F(D^2 \ol u)= \ol g.
\end{equation*}
By direct computation, 
\begin{equation*}
\sup_{Q_1} |\ol u| \leq \sup_{Q_1}|u| + \sup_{-1\leq t \leq 0} |t| |g(0,0)|    \leq \sup_{Q_1}|u| + |g(0,0)|.
\end{equation*}
Then we reduce to the case that $F(0)=0$. If $F(0) \neq 0$, then, by ellipticity, there exists $a\in \R$ such that 
\begin{equation}\label{est_F_a}
F(aI_d)=0 \quad \text{with } |a| \leq \frac{1}{\lambda d}|F(0)|. 
\end{equation}
Define the operator $ \widehat F(M):=F(M+a I_d)$ and observe that $\widehat F$ satisfies \ref{f1} with the same ellipticity constants 
$\lambda, \Lambda$ and $ \widehat F(0)= F(a I_d)=0$. It is clear that $\widehat u(x,t):= u(x,t) - \tfrac{1}{2}a |x|^2$ is a solution of  
\begin{equation*}
 \partial_t \widehat u + \widehat F(D^2 \hat u)=0
\end{equation*}
By direct computation, 
\begin{equation*}
\sup_{Q_1} |\hat u| \leq \sup_{Q_1}|u| + \tfrac{1}{2}|a| \sup_{x\in B_1} |x|^2  
                    \leq \sup_{Q_1}|u| + \tfrac{1}{2}|a|.
\end{equation*}
By applying \eqref{est_w3_scale_with_sup} and using \eqref{est_F_a}, we obtain
\begin{equation*}
 \left| \left\{ (x,t)\in Q_{1/2}\left(0,-\tfrac{1}{4}\right) : \Psi( u,Q_{3/4})(x,t)>  \kappa \right\} \right| 
 \leq C \left( \frac{\kappa}{\sup_{Q_1}|u| + \frac{1}{2\lambda d}|F(0)|} \right)^{-\eps}, 
\end{equation*}
and we get the inequality given by Theorem~\ref{W3_eps_para}.

Next we reduce to the case that $\sup_{Q_1}|u|\leq 1$. 
Assume that we have shown if $\sup_{Q_1}|u|\leq 1$ and $F(0)=0$, then for all $\kappa>0$, 
\begin{equation}\label{est_w3_scale}
 \left| \left\{ (x,t)\in Q_{1/2}\left(0,-\tfrac{1}{4}\right) : \Psi(u,Q_{3/4})(x,t)> \kappa \right\} \right|
 \leq C \left( \frac{\kappa}{1+\left\|g \right\|_{C^{0,1}(Q_1)}} \right)^{-\eps}. 
\end{equation}
We claim that if $F(0)=0$, and $\s:=\sup_{Q_1}|u|>0$ then for all $\kappa>0$, 
\begin{equation}\label{est_w3_scale_with_sup}
 \left| \left\{ (x,t)\in Q_{1/2}\left(0,-\tfrac{1}{4}\right) : \Psi( u,Q_{3/4})(x,t)>  \kappa \right\} \right| 
 \leq C \left( \frac{\kappa}{\s +\left\|g \right\|_{C^{0,1}(Q_1)}} \right)^{-\eps}. 
\end{equation}
Define the function $\tilde g:= (1/\s) g$, the operator $\widetilde F(M):= \s^{-1}F(\s M)$ and observe 
 $\tilde g\in C^{0,1}(Q_1)$ with  $\left\|\tilde g \right\|_{C^{0,1}(Q_1)} = \left\|g \right\|_{C^{0,1}(Q_1)}/\s$,
$\widetilde F$ satisfies \ref{f1} with the same ellipticity constants 
$\lambda, \Lambda$ and $\widetilde F(0)= \s^{-1}F(0)=0$. It is clear that $\tilde u:=u/\s$ is a solution of  
\begin{equation*}
 \partial_t \tilde u +\widetilde F(D^2\tilde u)=\tilde g  
 \end{equation*}
with $\sup_{Q_1} |\tilde u|=1$. By applying \eqref{est_w3_scale} to $\tilde u$, we obtain that, for all $\kappa>0$,
\begin{equation*}
 \left| \left\{ (x,t)\in Q_{1/2}\left(0,-\tfrac{1}{4}\right) : \Psi(\tilde u,Q_{3/4})(x,t)> \kappa \right\} \right| 
 \leq C \left( \frac{\kappa}{1+\left\|g \right\|_{C^{0,1}(Q_1)}/\s} \right)^{-\eps}. 
\end{equation*}
Noticing that $\Psi(\tilde u,Q_{3/4})(x,t)=\tfrac{1}{\beta}  \Psi(u ,Q_{3/4})(x,t)$, we obtain that, for all $\kappa>0$, 
\begin{equation*}
 \left| \left\{ (x,t)\in Q_{1/2}\left(0,-\tfrac{1}{4}\right) : \Psi( u,Q_{3/4})(x,t)> \s \kappa \right\} \right| 
 \leq C \left( \frac{\kappa}{1+\left\|g \right\|_{C^{0,1}(Q_1)}/\s} \right)^{-\eps}. 
\end{equation*}
This is equivalent to~\eqref{est_w3_scale_with_sup}. 
\end{step}

\begin{step} \label{step2_th2}
Assuming that $\sup_{Q_1}|u| \leq 1$, $g(0,0)=0$ and $F(0)=0$, we give the proof of~\eqref{est_w3_scale}. 
It suffices to get the inequality for $\kappa\geq \kappa_1$, where $\kappa_1$ is a universal constant. 
According to Proposition~\ref{int_reg_parabolic}, $Du$ is continuous and there exists a constant $C$ such that
\begin{equation} \label{est_inter_blackbox}
 \sup_{Q_{3/4}}|Du| \leq C \left(1+\left\|g \right\|_{C^{0,1}(Q_1)} \right). 
\end{equation}
Moreover, we claim that for every unit direction $e\in \R^d$, $|e|=1$, the function $u_e=e \cdot Du$ satisfies the inequalities
\begin{equation*}
\partial_t u_e+\mathcal{P}^-_{\lambda,\Lambda}(u_e) -\left\|g \right\|_{C^{0,1}(Q_1)}
\leq 0 \leq \partial_t u_e+\mathcal{P}^+_{\lambda,\Lambda}(u_e) +\left\|g \right\|_{C^{0,1}(Q_1)},  \quad \text{ in }Q_1,
\end{equation*}
 in the viscosity sense. 
We refer to~\cite[Lemma~3.12]{caff_cabre} for the elliptic version of this statement which is easy to generalize to the parabolic setting. \\
 According to Proposition~\ref{W_2eps_para}, we have, for each $\kappa>0$, 
\begin{equation*} 
 |\{  (x,t) \in Q_{1/2}\left(0, -\tfrac{1}{4} \right) : \udl \Theta(u_e,Q_{3/4}) (x,t)>\kappa \}| 
 \leq C \left(\frac{\kappa}{\sup_{Q_{3/4}} |u_e| +\left\|g \right\|_{C^{0,1}(Q_1)}} \right)^{-\eps}, 
\end{equation*}
where $C, \eps>0$ are universal constants. 
Thus, we deduce from~\eqref{est_inter_blackbox} that there exists a new universal constant~$C>0$ such that for all $\kappa>0$, 
\begin{equation} \label{w2eps_space_derivative}
 |\{  (x,t) \in Q_{1/2}\left(0, -\tfrac{1}{4} \right): \udl \Theta(u_e,Q_{3/4}) (x,t)>\kappa \}| 
 \leq C \left( \frac{\kappa}{1+\left\|g \right\|_{C^{0,1}(Q_1)}} \right)^{-\eps}.
\end{equation}
By Proposition~\ref{key_prop_para_ref}, there exists a universal constant $C_1$ such that, 
for all $\kappa> C_1 \left(1+\left\|g \right\|_{C^{0,1}(Q_1)} \right)$, 
\begin{multline*}
\left| \left\{ (x,t)\in Q_{1/2}\left(0, -\tfrac{1}{4} \right) : \Psi(u,Q_{3/4}) (x,t)>\kappa \right\} \right|  \\
\leq \sum_{i=1}^d \left|\left\{ (x,t)\in Q_{1/2}\left(0,-\tfrac{1}{4} \right):
\Theta (u_{x_i},Q_{3/4})(x,t)>  \frac{\kappa}{ C_1 \left(1+\left\|g \right\|_{C^{0,1}(Q_1)} \right) \sqrt{d}}  \right\} \right|, 
\end{multline*}
and we obtain the desired result for $\kappa>C_1 \left(1+\left\|g \right\|_{C^{0,1}(Q_1)}\right)$ by applying~\eqref{w2eps_space_derivative}. 
This completes the proof of~\eqref{est_w3_scale}.\qedhere
\end{step}
\end{proof}

\bibliography{ref_partial_reg}\bibliographystyle{plain}

\end{document}